\documentclass[10pt,a4paper]{amsart}
\usepackage{graphicx}
\usepackage{amssymb}
\usepackage{amsmath}
\usepackage{amsthm}
\usepackage{amscd}
\usepackage[all,2cell]{xy}
\usepackage[markup=underlined]{changes}
\definechangesauthor[name=zy, color=cyan]{zy}

\usepackage[pagebackref,colorlinks]{hyperref}

\UseAllTwocells \SilentMatrices
\newtheorem{thm}{Theorem}[section]

\newtheorem{cor}[thm]{Corollary}
\newtheorem{lem}[thm]{Lemma}
\newtheorem{exm}[thm]{Example}

\newtheorem{ques}[thm]{Question}

\newtheorem{conj}[thm]{Conjecture}
\newtheorem{prop}[thm]{Proposition}
\theoremstyle{definition}
\newtheorem{defn}[thm]{Definition}
\theoremstyle{remark}
\newtheorem{rem}[thm]{\bf Remark}
\numberwithin{equation}{section}

\begin{document}
\title[Pre-weight struct, pseudo-identities and canonical derived equiv]{Pre-weight structures, pseudo-identities and canonical derived equivalences}
\author[Xiao-Wu Chen] {Xiao-Wu Chen}

\date{\today}

\maketitle
\date{}%
\dedicatory{}%
\commby{}%
\begin{center}
\end{center}

\begin{abstract}
We introduce the notion of pre-weight structure on a triangulated category and study the corresponding pseudo-identities. We propose the notion of canonical derived equivalence between algebras that are not necessarily flat, which is associated to a tilting complex. In the flat situation, canonical derived equivalences coincide with standard derived equivalences in the sense of Rickard. We  prove that any derived equivalence starting from a hereditary algebra is canonical. The key tool is a general factorization theorem: any derived equivalence is uniquely factorized as a pseudo-identity followed by a canonical derived equivalence.
\end{abstract}

\section{Introduction}
Let $\mathbb{K}$ be a commutative ring, and let $A$ be an algebra over $\mathbb{K}$. Denote by $\mathbf{D}(A\mbox{-Mod})$ the derived category of unbounded complexes of left $A$-modules. For another algebra $B$, a linear triangle equivalence $F\colon \mathbf{D}(A\mbox{-Mod})\rightarrow \mathbf{D}(B\mbox{-Mod})$ is called a \emph{derived equivalence} between $A$ and $B$. It is a central object both in the derived Morita theory \cite{Ric89, Toen} and representation theory of algebras \cite{Zim, Xi}.

Assume that both $A$ and $B$ are flat over $\mathbb{K}$. The derived equivalence $F$ above is called \emph{standard} \cite{Ric91} if it is isomorphic to a derived tensor functor $X\otimes_A^\mathbb{L}-$ for a complex $X$ of $B$-$A$-bimodules. Such a complex $X$ is called a \emph{two-sided tilting complex}. We mention that two-sided tilting complexes are derived analogues of invertible bimodules in the classical Morita theory. They play a central role in the modular representation theory of finite groups \cite{Ric96, Rou}. Derived equivalences of Fourier-Mukai type \cite{Huy} between projective schemes are geometric analogues of standard derived equivalences between algebras, and Fourier-Mukai kernels correspond to two-sided tilting complexes. It is a well-known open question raised in \cite{Ric91} whether any derived equivalence between flat algebras is standard.

We are interested in derived equivalences between algebras that are not necessarily flat. One motivation is the study of derived equivalences \cite{Ric89} between general rings, or equivalently,  algebras over $\mathbb{Z}$ which are not necessarily flat. Another reason for this general study is the striking fact that derived equivalences do not preserve the flatness of the algebras in general. Such examples appear naturally in the tilting theory for hereditary orders \cite{KZ}; see also \cite[Subsection~6.7.1]{Zim} and \cite{ChenYP, Xi}.

It is well known that in this general setting, derived tensor products of complexes of bimodules are not well-behaved. Therefore, standard derived equivalences are not the right object to study. The fundamental problem is as follows: what kind of derived equivalences paly the role of standard derived equivalences in this general setting?

We propose the notion of \emph{canonical derived equivalence} between arbitrary algebras. To be more precise, let $A$ and $B$ be two algebras over $\mathbb{K}$ that are not necessarily flat. For a tilting complex $P$ over $B$ and an algebra isomorphism $\phi\colon A\rightarrow {\rm End}_{\mathbf{D}(B\mbox{-}{\rm Mod})}(P)^{\rm op}$, we consider the dg endomorphism algebra ${\rm End}_B(P)$,  and set $\Gamma=(\tau_{\leq 0}\; {\rm End}_B(P))^{\rm op}$ to be the opposite dg algebra of the truncated subalgebra of ${\rm End}_B(P)$. We have a surjective quasi-isomorphism $\Gamma \rightarrow A$ by composing $\phi^{-1}$ with the projection $\Gamma\rightarrow {\rm End}_{\mathbf{D}(B\mbox{-}{\rm Mod})}(P)^{\rm op}$.

We have the following derived equivalence associated to $(P, \phi)$.
$$\Psi_{(P, \phi)}\colon \mathbf{D}(A\mbox{-Mod}) \longrightarrow  \mathbf{D}(\Gamma)\xrightarrow{P\otimes^\mathbb{L}_\Gamma-} \mathbf{D}(B\mbox{-Mod})$$
Here, $\mathbf{D}(\Gamma)$ denotes the derived category \cite{Kel94} of left dg $\Gamma$-modules; the unnamed arrow above is induced by the surjective quasi-isomorphism $\Gamma\rightarrow A$ above. We call a derived equivalence \emph{canonical} if it is isomorphic to $\Psi_{(P, \phi)}$ for some pair $(P, \phi)$.

  We emphasize that the pair $(P, \phi)$ does not make $P$ into a complex of $B$-$A$-bimodules, since $A$ acts on $P$ only up to homotopy; compare \cite[Section~4]{CPS} and \cite{Kel00}. The construction of $\Psi_{(P,\phi)}$ is implicitly contained in \cite{Kel93} at least for  flat algebras.

The following theorem is inspired by \cite{Kel93} and  justifies this new notion to some extent; see Theorem~\ref{thm:stan=can}.
\vskip 5pt

\noindent {\bf Theorem A.}\; \emph{Let $A$ and $B$ be flat algebras. Then a derived equivalence $\mathbf{D}(A\mbox{-}{\rm Mod})\rightarrow \mathbf{D}(B\mbox{-}{\rm Mod})$ is standard if and only if it is canonical. }

\vskip 5pt

In view of the open question above and Theorem~A, we propose the following more general conjecture.

\vskip 5pt

\noindent {\bf Conjecture B.}\; \emph{Let $A$ and $B$ be arbitrary algebras. Then any derived equivalence $\mathbf{D}(A\mbox{-}{\rm Mod})\rightarrow \mathbf{D}(B\mbox{-}{\rm Mod})$ is canonical.}

\vskip 5pt

Recall that pseudo-identities on bounded homotopy categories and bounded derived categories are introduced in \cite{CY}. Such endofunctors are almost identical to the genuine identity functor. In this work, we study a \emph{pseudo-identity} on the \emph{unbounded} derived category.

The following general factorization theorem implies that the core of Conjecture~B is to understand pseudo-identities on unbounded derived categories; see Theorem~\ref{thm:factor}.  We mention that the idea of such a factorization is already implicit  in \cite{Ric89, Ric91}; compare  \cite{CY, CC}.
\vskip 5pt

\noindent {\bf Theorem C.}\;  \emph{Let $A$ and $B$ be arbitrary algebras. Then any derived equivalence $F\colon \mathbf{D}(A\mbox{-}{\rm Mod})\rightarrow \mathbf{D}(B\mbox{-}{\rm Mod})$ admits a unique factorization $F\simeq F_2F_1$ with $F_1$ a pseudo-identity on $\mathbf{D}(A\mbox{-}{\rm Mod})$ and $F_2\colon \mathbf{D}(A\mbox{-}{\rm Mod})\rightarrow \mathbf{D}(B\mbox{-}{\rm Mod})$ a canonical derived equivalence. }

\vskip 5pt

Using the factorization in Theorem~C, we confirm Conjecture~B  when the algebra $A$ is left hereditary; see Theorem~\ref{thm:hereditary+can}.
\vskip 5pt

\noindent {\bf Theorem D}. \; \emph{Let $A$ and $B$ be two algebras with $A$ being left hereditary. Then any derived equivalence $F\colon \mathbf{D}(A\mbox{-}{\rm Mod})\rightarrow \mathbf{D}(B\mbox{-}{\rm Mod})$ is canonical.}
\vskip 5pt

We observe that pseudo-identites \cite{CY} on the homotopy categories and derived categories rely on a certain inner structure of the categories in consideration. Inspired by this observation, we introduce the notion of \emph{pre-weight structure} on a general triangulated category, which contains weight structures in \cite{Bon} (also called co-t-structures in \cite{Pau}); see Section~\ref{sec:pws}. We introduce \emph{pseudo-identities} with respect to a given pre-weight structure on a triangulated category; see Section~\ref{sec:ps}. This general notion unifies pseudo-identities both on the  homotopy categories and  derived categories in \cite{CY}.

Let us describe the structure of the paper.

In Section~\ref{sec:pws}, we introduce the notion of pre-weight structure on a triangulated category and study its basic properties. In Section~\ref{sec:conv-dim}, we introduce convergence conditions and dimension for a pre-weight structure. The convergence conditions are used to establish a vanishing result in Proposition~\ref{prop:vanishing}. We prove that a pre-weight structure induces a semi-orthogonal decomposition of a certain Verdier quotient category; see Proposition~\ref{prop:semi-ortho}. In Section~\ref{sec:exm}, we study the canonical pre-weight structures both on the homotopy category and derived category of complexes.

In Section~\ref{sec:triv-act}, we give sufficient conditions on when a triangle endofunctor acts trivially on objects in Theorem~\ref{thm:trivial-action}. Under certain conditions, we prove in Theorem~\ref{thm:extend-ff} that a triangle functor is fully faithful if so is its restriction to the bounded part.

We introduce the notion of pseudo-identity on a triangulated category with respect to a pre-weight structure in Section~\ref{sec:ps}. Based on results in Section~\ref{sec:triv-act}, we give in Theorem~\ref{thm:pseudo} sufficient conditions on when a triangle endofunctor is isomorphic to a pseudo-identity. We prove that any pseudo-identity on the derived category of a hereditary abelian category is isomorphic to the genuine identity functor; see Proposition~\ref{prop:hereditary+ps}.

We study the dg endomorphism algebras of complexes in Section~\ref{sec:dg}. We introduce the notion of canonical derived equivalence and  prove Theorem~C in Section~\ref{sec:can}.  We give a criterion on when two canonical derived equivalences are isomorphic; see Theorem~\ref{thm:can}.  We prove Theorem~D which confirms Conjecture~B for   left hereditary algebras. In Section~\ref{sec:stan}, we prove Theorem~A. In Proposition~\ref{prop:hereditary+stan}, we prove that any derived equivalence between flat hereditary algebras is standard.

We  will denote by $\Sigma$ the suspension functor of any triangulated category. As usual, we abbreviate `differential graded' as `dg'. By default, a module means a left module. In the last three sections, we will work over a fixed commutative ring $\mathbb{K}$.

\section{Pre-weight structures}\label{sec:pws}

In this section, we introduce the notion of pre-weight structure on a general triangulated category. In Proposition~\ref{prop:tower}, we establish the existence of  left towers and right towers for certain objects. We refer to Section~\ref{sec:exm} for the motivating examples. In Section~\ref{sec:ps}, we will introduce pseudo-identities on a triangulated category equipped with a pre-weight structure.

Let $\mathcal{T}$ be a triangulated category with $\Sigma$ its suspension functor. Denote by $\Sigma^{-1}$ a quasi-inverse of $\Sigma$. Consequently, $\Sigma^n$ is defined for each integer $n$. Let $\mathcal{X}$  be a full subcategory of $\mathcal{T}$. We say that $\mathcal{X}$ is closed under $\Sigma$ (\emph{resp}. $\Sigma^{-1}$) if $\Sigma(\mathcal{X})\subseteq \mathcal{X}$ (\emph{resp}. $\Sigma^{-1}(\mathcal{X})\subseteq \mathcal{X}$). Denote by ${\rm tri}\langle \mathcal{X}\rangle$ the smallest triangulated subcategory of $\mathcal{T}$ containing $\mathcal{X}$.

Let $\mathcal{Y}$ be another full subcategory. We denote by $\mathcal{X}\ast\mathcal{Y}$ the full subcategory formed by those objects $Z$ which fit into an exact triangle $X\rightarrow Z\rightarrow Y \rightarrow \Sigma(X)$ for some $X\in \mathcal{X}$ and $Y\in \mathcal{Y}$. The full subcategory $\mathcal{X}$ is said to be closed under extensions if $\mathcal{X}\ast \mathcal{X}\subseteq \mathcal{X}$.

\begin{rem}\label{rem:ast}
By the octahedral axiom (TR4), the operation ``$\ast$" on full subcategories is associative; see \cite[Lemme~1.3.10]{BBD}. The following fact will be useful: given full subcategories $\mathcal{X}_0, \mathcal{X}_1, \cdots, \mathcal{X}_n$ of $\mathcal{T}$, an object $E$ lies in $\mathcal{X}_0\ast \mathcal{X}_1\ast \cdots \ast  \mathcal{X}_n$ if and only if there is a sequence of morphisms $$E_0\rightarrow E_1\rightarrow \cdots \rightarrow E_{n-1}\rightarrow E_{n}=E$$
such that $E_0\in \mathcal{X}_0$ and that the cone of each morphism $E_{i}\rightarrow E_{i+1}$ lies in $\mathcal{X}_{i+1}$. Moreover, we observe that each  $E_i$ in the sequence above belongs to $\mathcal{X}_0\ast \mathcal{X}_1\ast \cdots \ast \mathcal{X}_i$.
\end{rem}

\begin{defn}\label{defn:pws}
A pair $(\mathcal{U}_{\geq 0}, \mathcal{U}_{\leq 0})$ of full additive subcategories in $\mathcal{T}$  is called  a \emph{pre-weight structure} on $\mathcal{T}$, if the following conditions are satisfied:
\begin{enumerate}
\item $\mathcal{U}_{\geq 0}\ast \Sigma(\mathcal{U}_{\leq 0})=\mathcal{T}$;
\item $\mathcal{U}_{\geq 0}$ is closed under $\Sigma^{-1}$ and extensions;
\item $\mathcal{U}_{\leq  0}$ is closed under $\Sigma$ and extensions.
\end{enumerate}
The full subcategory $\mathcal{C}=\mathcal{U}_{\geq 0}\cap \mathcal{U}_{\leq 0}$ is called the \emph{core} of the pre-weight structure.
\end{defn}

We observe that a \emph{weight structure} in the sense of \cite[Definition~1.1.1]{Bon} is a pre-weight structure. This observation justifies our terminology to some extent. We mention that a weight structure is also called a co-t-structure in \cite[Definition~2.4]{Pau}.

For each integer $m$, we set $\mathcal{U}_{\geq m}=\Sigma^{-m}(\mathcal{U}_{\geq 0})$ and $\mathcal{U}_{\leq  m}=\Sigma^{-m}(\mathcal{U}_{\leq 0})$. Then Definition~\ref{defn:pws}(1) is rewritten as $$\mathcal{U}_{\geq 0}\ast \mathcal{U}_{\leq -1}=\mathcal{T}.$$ For $m\leq n$, we set $\mathcal{U}_{[m, n]}=\mathcal{U}_{\geq m}\cap \mathcal{U}_{\leq n}$.  In particular, we have $\mathcal{C}=\mathcal{U}_{[0, 0]}$. We set $\mathcal{U}_{+}=\bigcap_{m\in \mathbb{Z}} \mathcal{U}_{\geq m}$, $\mathcal{U}_{-}=\bigcap_{m\in \mathbb{Z}} \mathcal{U}_{\leq m}$  and $\mathcal{U}_b=\bigcup_{m\leq n}\mathcal{U}_{[m, n]}$. We observe that
$$\mathcal{U}_b=\mathcal{U}_{+}\cap \mathcal{U}_{-}.$$
 Since $\mathcal{U}_+$ is closed under $\Sigma^{\pm}$ and extensions, it is a triangulated subcategory of $\mathcal{T}$. Similarly, both $\mathcal{U}_{-}$ and $\mathcal{U}_b$ are triangulated subcategories of $\mathcal{T}$. The subcategory $\mathcal{U}_b$ might be viewed as the \emph{bounded part} of $\mathcal{T}$, or of the pre-weight structure.

\begin{lem}\label{lem:pws}
Let $(\mathcal{U}_{\geq 0}, \mathcal{U}_{\leq 0})$  be a pre-weight structure on $\mathcal{T}$ with core $\mathcal{C}$. Then the following statements hold.
\begin{enumerate}
\item $\mathcal{U}_{\geq m}=\mathcal{U}_{> n}\ast \mathcal{U}_{[m, n]}$ for any $m\leq n$.
\item $\mathcal{U}_{[m, p]}=\mathcal{U}_{[n+1, p]}\ast \mathcal{U}_{[m, n]}$ for any $m\leq n\leq p-1$.
\item $\mathcal{U}_{\leq  n}=\mathcal{U}_{[m, n]}\ast \mathcal{U}_{<m}$ for any $m\leq n$.
\item $\mathcal{U}_{[q, n]}=\mathcal{U}_{[m, n]}\ast \mathcal{U}_{[q, m-1]}$ for any $q+1\leq m\leq n$.
\item $\mathcal{U}_{[m, n]}=(\Sigma^{-n}\mathcal{C})\ast  (\Sigma^{1-n}\mathcal{C})\ast \cdots \ast (\Sigma^{-m}\mathcal{C})$ for any $m\leq n$.
\item ${\rm tri}\langle \mathcal{C} \rangle=\mathcal{U}_b$.
\end{enumerate}
\end{lem}

\begin{proof}
For (1), we observe that $\mathcal{U}_{> n}\subseteq \mathcal{U}_{\geq m}$ and $\mathcal{U}_{[m, n]}\subseteq \mathcal{U}_{\geq m}$. Since $\mathcal{U}_{\geq m}$ is closed under extensions, we have
$\mathcal{U}_{> n}\ast \mathcal{U}_{[m, n]}\subseteq \mathcal{U}_{\geq m}$. Conversely, take any object $X\in \mathcal{U}_{\geq m}$. Applying $\Sigma^{-(n+1)}$ to Definition~\ref{defn:pws}(1), we obtain $\mathcal{U}_{>n}\ast \mathcal{U}_{\leq n}=\mathcal{T}$. It follows that there exists an exact triangle
$$X_1\longrightarrow X\longrightarrow X_2\longrightarrow \Sigma(X_1)$$
with $X_1\in \mathcal{U}_{>n}$ and $X_2\in \mathcal{U}_{\leq n}$. Both $X$ and $\Sigma(X_1)$ belong to $\mathcal{U}_{\geq m}$. Since $\mathcal{U}_{\geq m}$ is closed under extensions, we infer  $X_2\in \mathcal{U}_{\geq m}$. Therefore, $X_2$ belongs to $\mathcal{U}_{[m, n]}$. This proves that $X\in \mathcal{U}_{>n}\ast \mathcal{U}_{[m, n]}$, as required.

The proof of (2) is similar to the one of (1). Dually, we have (3) and (4).

For (5), we observe from (2) that $\mathcal{U}_{[m, n]}=\mathcal{U}_{[m+1, n]}\ast \Sigma^{-m}(\mathcal{C})$ for $m<n$. Then we deduce (5) by induction on $n-m$.

For (6), we recall that $\mathcal{U}_b$ is a triangulated subcategory of $\mathcal{T}$. Since $\mathcal{C}\subseteq \mathcal{U}_b$, it follows that ${\rm tri}\langle \mathcal{C} \rangle\subseteq \mathcal{U}_b$. On the other hand, we infer $\mathcal{U}_b=\bigcup_{m\leq n}\mathcal{U}_{[m, n]} \subseteq {\rm tri}\langle \mathcal{C}\rangle$ from (5).
\end{proof}

The following consequence follows immediately from Remark~\ref{rem:ast} and Lemma~\ref{lem:pws}(5).

\begin{cor}\label{cor:ast}
Let $X\in \mathcal{U}_{[-n, 0]}$ for some $n\geq 0$. Then there exists a sequence of morphisms
$$X_0\stackrel{\phi_0}\longrightarrow X_1\longrightarrow \cdots \longrightarrow X_{n-1} \stackrel{\phi_{n-1}}\longrightarrow X_{n}=X$$
such that each $X_i\in \mathcal{U}_{[-i, 0]}$ and that the cone of $\phi_i$ lies in $\Sigma^{i+1}(\mathcal{C})$. \hfill $\square$
\end{cor}

There is another immediate consequence of Lemma~\ref{lem:pws}(5), which be will used later.

\begin{cor}\label{cor:vanishing}
Assume that the core $\mathcal{C}$ satisfies ${\rm Hom}_\mathcal{T}(\mathcal{C}, \Sigma^{-n}(\mathcal{C}))=0$ for any $n\geq 1$. Then we have ${\rm Hom}_{\mathcal{T}}(\mathcal{U}_{[-n,0]}, \mathcal{U}_{[1, m]})=0$ for any $n\geq 0$ and $m\geq 1$. \hfill $\square$
\end{cor}

The first statement in the following result might be viewed as an infinite version of Corollary~\ref{cor:ast}.

\begin{prop}\label{prop:tower}
Let $(\mathcal{U}_{\geq 0}, \mathcal{U}_{\leq 0})$  be a pre-weight structure on $\mathcal{T}$ with core $\mathcal{C}$. Assume that $X\in \mathcal{U}_{\leq 0}$ and $Y\in \mathcal{U}_{\geq 0}$. Then the following two statements hold.
\begin{enumerate}
\item There exist morphisms $\iota_n\colon X_n\rightarrow X$ and $\phi_n\colon X_n\rightarrow X_{n+1}$ for all $n\geq 0$ such that $X_n\in \mathcal{U}_{[-n, 0]}$, the cone of $\iota_n$ lies in $\mathcal{U}_{\leq -n-1}$,  the cone of $\phi_n$ lies in $\Sigma^{n+1}(\mathcal{C})$ and that $\iota_n=\iota_{n+1}\circ \phi_n$.
\item There exist morphisms $\pi_n\colon Y\rightarrow Y_n$ and $\psi_n\colon Y_{n+1}\rightarrow Y_n$ for all $n\geq 0$ such that $Y_n\in \mathcal{U}_{[0, n]}$, the cone of $\pi_n$ lies in $\mathcal{U}_{\geq n}$,  the cone of $\psi_n$ lies in $\Sigma^{-n}(\mathcal{C})$ and that $\pi_n=\psi_n\circ \pi_{n+1}$.
\end{enumerate}
\end{prop}

We will call the data $(\iota_n, \phi_n)_{n\geq 0}$ above a \emph{left tower} of $X$, and $(\pi_n, \psi_n)_{n\geq 0}$ a \emph{right tower} of $Y$. They are not unique, in general.

\begin{proof}
By duality, we only prove (1). By Lemma~\ref{lem:pws}(3), we have $\mathcal{U}_{\leq 0}=\mathcal{C}\ast \mathcal{U}_{\leq -1}$ and $\mathcal{U}_{\leq -1}=\Sigma(\mathcal{C})\ast \mathcal{U}_{\leq -2}$. Then we have two exact triangles
$$X_0\stackrel{\iota_0}\longrightarrow X\longrightarrow E \longrightarrow \Sigma(X_0) \mbox{  and  } E_0\longrightarrow E\longrightarrow F \longrightarrow \Sigma(E_0), $$
such that $X_0\in \mathcal{C}$,  $E\in \mathcal{U}_{\leq -1}$, $E_0\in \Sigma(\mathcal{C})$ and $F\in \mathcal{U}_{\leq -2}$. By the octahedral axiom (TR4), we have the following commutative diagram such that the upper row and the second column are exact triangles.
\[\xymatrix{
X_0\ar@{=}[d] \ar@{.>}[r]^-{\phi_0} & X_1 \ar@{.>}[d]^-{\iota_1} \ar@{.>}[r] & E_0\ar[d] \ar[r]& \Sigma(X_0) \ar@{=}[d]\\
X_0\ar[r]^-{\iota_0} & X\ar[d] \ar[r]&  E\ar[d] \ar[r] & \Sigma(X_0)\ar[d]^-{\Sigma(\phi_0)}\\
                     & F\ar[d] \ar@{=}[r] & F\ar[d]  \ar[r] & \Sigma(X_1)\\
                     & \Sigma(X_1) \ar[r] & \Sigma(E_0)
}\]
Thus we have constructed the required morphisms  $\iota_0$, $\iota_1$ and $\phi_0$. We iterate the argument above to the exact triangle $X_1\stackrel{\iota_1}\rightarrow X\rightarrow F \rightarrow \Sigma(X_1)$ and the fact that $F\in \mathcal{U}_{\leq -2}=\Sigma^2(\mathcal{C})\ast \mathcal{U}_{\leq -3}$. By induction, we complete the proof.
\end{proof}

The following observation will be used later.

\begin{lem}\label{lem:factor}
Let $(\mathcal{U}_{\geq 0}, \mathcal{U}_{\leq 0})$  be a pre-weight structure on $\mathcal{T}$. Assume that $X\in \mathcal{U}_+$ and $Y\in \mathcal{U}_{-}$. Then any morphism $f\colon X\rightarrow Y$ factors through some object in $\mathcal{U}_b$.
\end{lem}

\begin{proof}
Recall that both $\mathcal{U}_-$ and $\mathcal{U}_+$ are triangulated subcategories of $\mathcal{T}$. By Definition~\ref{defn:pws}(1), we infer that $\mathcal{U}_+\ast \mathcal{U}_{-}=\mathcal{T}$. Then the statement follows immediately from \cite[Lemma~1.1]{JK}.
\end{proof}

 \section{Convergence and dimension} \label{sec:conv-dim}

 In this section, we introduce convergence conditions and dimension for a pre-weight structure.  Using the convergence conditions, we obtain a vanishing property of the pre-weight structure in Proposition~\ref{prop:vanishing}. We show that a pre-weight structure induces a semi-orthogonal decomposition on a certain Verdier quotient category; see Proposition~\ref{prop:semi-ortho}.

 We fix a triangulated category $\mathcal{T}$, which is endowed with a pre-weight structure $(\mathcal{U}_{\geq 0}, \mathcal{U}_{\leq 0})$ with core $\mathcal{C}$.

We assume that $\mathcal{T}$ is cocomplete, that is,  has arbitrary coproducts. Consider an object $X\in \mathcal{U}_{\leq 0}$ and its left tower $(\iota_n, \phi_n)_{n\geq 0}$. The morphisms $\phi_n$'s yield the following exact triangle
$$\coprod_{n\geq 0}X_n \stackrel{1-\phi}\longrightarrow \coprod_{n\geq 0} X_n \stackrel{\gamma}\longrightarrow {\rm hocolim}\; X_n\stackrel{\delta}\longrightarrow \Sigma(\coprod_{n\geq 0} X_n).$$
Here, the morphism $1-\phi$ is uniquely determined by the condition that its composition with the inclusion $X_n\rightarrow \coprod_{n\geq 0} X_n$ is given by
$$X_n \xrightarrow{\begin{pmatrix}
{\rm Id}_{X_n}\\ -\phi_n\end{pmatrix}}   X_n\oplus X_{n+1}\longrightarrow\coprod_{n\geq 0} X_n.$$
The object ${\rm hocolim}\; X_n$ is called a \emph{homotopy colimit} of the sequence $\{\phi_n\}_{n\geq 0}$, which is unique up to a non-canonical isomorphism; see  \cite[Definition~2.1]{BN}. More precisely,  by referring to the homotopy colimit, we really mean the  triple
$$({\rm hocolim}\; X_n; \gamma, \delta).$$
 We denote by $$\gamma_n\colon X_n\longrightarrow {\rm hocolim}\; X_n$$
 the restriction of $\gamma$ to the $n$-th component.

\begin{defn}\label{defn:left-conv}
The left tower $(\iota_n, \phi_n)_{n\geq 0}$ above is called \emph{convergent} if there is an isomorphism $\kappa\colon {\rm hocolim}\; X_n\rightarrow X$ such that $\kappa\circ \gamma_n=\iota_n$ for each $n\geq 0$. The pre-weight structure $(\mathcal{U}_{\geq 0}, \mathcal{U}_{\leq 0})$  is called \emph{left convergent} if each object in $\mathcal{U}_{\leq 0}$ admits a convergent left tower.
\end{defn}

\begin{lem}\label{lem:vanishing}
Suppose that $\mathcal{T}$ is cocomplete and that the pre-weight structure $(\mathcal{U}_{\geq 0}, \mathcal{U}_{\leq 0})$ is left convergent. Assume that ${\rm Hom}_\mathcal{T}(\mathcal{C}, \Sigma^{-n}(\mathcal{C}))=0$ for any $n\geq 1$. Then we have ${\rm Hom}_{\mathcal{T}}(\mathcal{U}_{\leq 0}, \mathcal{U}_{[1, m]})=0$ for any $m\geq 1$.
\end{lem}

\begin{proof}
Fix any $X\in \mathcal{U}_{\leq 0}$ and $Y\in \mathcal{U}_{[1, m]}$. We will show that ${\rm Hom}_\mathcal{T}(X, Y)=0$. For this, we fix a left tower  $(\iota_n, \phi_n)_{n\geq 0}$ of $X$ which is convergent.  By identifying $X$ with ${\rm hocolim}\; X_n$, we have an exact triangle.
\begin{align}\label{tri:hocolim}
\coprod_{n\geq 0}X_n \stackrel{1-\phi}\longrightarrow \coprod_{n\geq 0} X_n \longrightarrow X \longrightarrow \Sigma(\coprod_{n\geq 0} X_n)
\end{align}
Recall that each $X_n$ belongs to $\mathcal{U}_{[-n, 0]}$. By Corollary~\ref{cor:vanishing}, we infer that
$${\rm Hom}_\mathcal{T}(\coprod_{n\geq 0} X_n, Y)=0={\rm Hom}_\mathcal{T}(\Sigma(\coprod_{n\geq 0} X_n), Y).$$
Applying ${\rm Hom}_\mathcal{T}(-, Y)$ to the exact triangle above, we infer   ${\rm Hom}_\mathcal{T}(X, Y)=0$.
\end{proof}

Dually, we assume that $\mathcal{T}$ is complete, that is, has arbitrary products. For an object $Y\in \mathcal{U}_{\geq 0}$ we consider its right tower $(\pi_n, \psi_n)_{n\geq 0}$. The sequence $\{\psi_n\}_{n\geq 0}$ of morphisms gives rise to its \emph{homotopy limit}
in the following exact triangle.
\begin{align}\label{tri:right-tower}
{\rm holim}\;Y_n\stackrel{\sigma}\longrightarrow \prod_{n\geq 0} Y_n\stackrel{1-\psi}\longrightarrow \prod_{n\geq 0} Y_n\stackrel{}\longrightarrow \Sigma({\rm holim}\;Y_n)
\end{align}
Here, $1-\psi$ is the unique morphism such that its composition with each projection $\prod_{n\geq 0} Y_n\rightarrow Y_n$ is given by
$$\prod_{n\geq 0} Y_n\longrightarrow Y_n\oplus Y_{n+1}\xrightarrow {\begin{pmatrix}{\rm Id}_{Y_n}, -\psi_n \end{pmatrix}} Y_n.$$
Denote by
$$\sigma_n\colon {\rm holim}\;Y_n\longrightarrow Y_n$$
the composition of $\sigma$ with the projection $\prod_{n\geq 0} Y_n\rightarrow Y_n$.

\begin{defn}\label{defn:right-conv}
The  right tower $(\pi_n, \psi_n)_{n\geq 0}$ is said to \emph{convergent} if there exists an isomorphism $\lambda\colon Y\rightarrow {\rm holim}\;Y_n$ such that $\sigma_n\circ \lambda=\pi_n$. The pre-weight structure $(\mathcal{U}_{\geq 0}, \mathcal{U}_{\leq 0})$  is called \emph{right convergent} if each object in $\mathcal{U}_{\geq 0}$ admits a convergent right tower.
\end{defn}

By duality, we have the following result.

\begin{lem}\label{lem:vanishing-dual}
Suppose that $\mathcal{T}$ is complete and that the pre-weight structure $(\mathcal{U}_{\geq 0}, \mathcal{U}_{\leq 0})$ is right convergent. Assume that ${\rm Hom}_\mathcal{T}(\mathcal{C}, \Sigma^{-n}(\mathcal{C}))=0$ for any $n\geq 1$. Then we have ${\rm Hom}_{\mathcal{T}}(\mathcal{U}_{[-n, -1]}, \mathcal{U}_{\geq 0})=0$ for any $n\geq 1$.\hfill $\square$
\end{lem}

Assume that  $\mathcal{T}$ is \emph{bi-complete}, that is, cocomplete and complete. We say that the pre-weight structure $(\mathcal{U}_{\geq 0}, \mathcal{U}_{\leq 0})$ is \emph{two-sided convergent}, if it is both left convergent and right convergent.

\begin{prop}\label{prop:vanishing}
Suppose that $\mathcal{T}$ is bi-complete and that the pre-weight structure $(\mathcal{U}_{\geq 0}, \mathcal{U}_{\leq 0})$ is two-sided convergent. Assume that ${\rm Hom}_\mathcal{T}(\mathcal{C}, \Sigma^{-n}(\mathcal{C}))=0$ for any $n\geq 1$. Then we have ${\rm Hom}_{\mathcal{T}}(\mathcal{U}_{\leq 0}, \mathcal{U}_{\geq 1})=0$.
\end{prop}

\begin{proof}
The following proof is almost the same with the one of Lemma~\ref{lem:vanishing}.  For any $X\in \mathcal{U}_{\leq 0}$ and $Y\in \mathcal{U}_{\geq 1}$, we have to show that ${\rm Hom}_\mathcal{T}(X, Y)=0$. Lemma~\ref{lem:vanishing-dual} implies that
$${\rm Hom}_\mathcal{T}(\coprod_{n\geq 0} X_n, Y)=0={\rm Hom}_\mathcal{T}(\Sigma(\coprod_{n\geq 0} X_n), Y).$$
Applying ${\rm Hom}_\mathcal{T}(-, Y)$ to the exact triangle (\ref{tri:hocolim}), we infer   ${\rm Hom}_\mathcal{T}(X, Y)=0$.
\end{proof}

The following notion is inspired by \cite[Definition~3.3]{CLZ}.

\begin{defn}
The \emph{dimension} of the pre-weight structure $(\mathcal{U}_{\geq 0}, \mathcal{U}_{\leq 0})$ on $\mathcal{T}$ is defined to be
$${\rm min}\{+\infty, d\geq -1\; |\; {\rm  Hom}_\mathcal{T}(\mathcal{U}_{\geq 0}, \mathcal{U}_{\leq -(d+1)})=0\}.$$
\end{defn}

The following fact is immediate for the very definition of the dimension.

\begin{lem}
Let $(\mathcal{U}_{\geq 0}, \mathcal{U}_{\leq 0})$  be a pre-weight structure on $\mathcal{T}$ with dimension $d$. Then $d\leq 0$ if and only if ${\rm Hom}_\mathcal{T}(\mathcal{U}_{\geq 0}, \mathcal{U}_{\leq 1})=0$. \hfill $\square$
\end{lem}

Consequently, the pair $(\mathcal{U}_{\geq 0}, \mathcal{U}_{\leq 0})$ is a weight structure in the sense of \cite[Definition~1.1.1]{Bon} if and only if it is a pre-weight structure of dimension $0$ or $-1$ such that both $\mathcal{U}_{\geq 0}$ and $\mathcal{U}_{\leq 0}$ are closed under direct summands.

We now strengthen Lemma~\ref{lem:factor}  when the pre-weight structure has finite dimension.

\begin{prop}\label{prop:pws-fd}
Let $(\mathcal{U}_{\geq 0}, \mathcal{U}_{\leq 0})$  be a pre-weight structure of finite dimension. Then for any $X\in \mathcal{U}_+$ and $Y\in \mathcal{U}_{-}$, there exist a sufficiently large $N$ and  two exact triangles $X'\rightarrow X\stackrel{a}\rightarrow X''\rightarrow \Sigma(X')$ and $Y'\stackrel{b}\rightarrow Y\rightarrow Y''\rightarrow \Sigma(Y')$ with $X'\in \mathcal{U}_{\geq N}$, $Y''\in \mathcal{U}_{\leq -N}$ and $X'', Y'\in \mathcal{U}_b$; moreover, the morphisms $a$ and $b$ induce an isomorphism
$${\rm Hom}_\mathcal{T}(X'', Y')\longrightarrow {\rm Hom}_\mathcal{T}(X, Y).$$
\end{prop}

\begin{proof}
Denote by $d$ the dimension of the pre-weight structure. We assume that $X\in \mathcal{U}_{\geq n}$ and $Y\in \mathcal{U}_{\leq m}$. Set $N={\rm max}\{n+d+2, m+d+2\}$. By Lemma~\ref{lem:pws}(1), we have an exact triangle
\begin{align}\label{tri:factor}
X'\longrightarrow X\stackrel{a}\longrightarrow X''\longrightarrow \Sigma(X')
\end{align}
with $X'\in \mathcal{U}_{\geq N}$ and $X''\in \mathcal{U}_{[n, N-1]}\subseteq \mathcal{U}_b$. Similarly, we have an exact triangle
 $$Y'\stackrel{b}\longrightarrow Y\longrightarrow Y''\longrightarrow \Sigma(Y')$$
 with $Y'\in \mathcal{U}_{[-N+1, m]}\subseteq \mathcal{U}_b$ and $Y''\in \mathcal{U}_{\leq -N}$.

 By the choice of $N$, we have
  $${\rm Hom}_\mathcal{T}(X', Y)=0={\rm  Hom}_\mathcal{T}(\Sigma(X'), Y).$$
  Using (\ref{tri:factor}), we infer  that $a$ induces an isomorphism $${\rm Hom}_\mathcal{T}(X'', Y)\longrightarrow {\rm Hom}_\mathcal{T}(X, Y).$$
 Similarly, the morphism $b$ induces an isomorphism
 $${\rm Hom}_\mathcal{T}(X'', Y')\longrightarrow {\rm Hom}_\mathcal{T}(X'', Y).$$
 Combining the two isomorphisms above, we complete the proof.
\end{proof}

Recall from \cite{Bondal, BK} that a \emph{semi-orthogonal decomposition} $\mathcal{T}=\langle \mathcal{Y}, \mathcal{X}\rangle$ consists of two triangulated subcategories $\mathcal{X}$ and $\mathcal{Y}$ satisfying $\mathcal{X}\ast\mathcal{Y}=\mathcal{T}$ and ${\rm Hom}_\mathcal{T}(\mathcal{X}, \mathcal{Y})=0$.  We mention that $\mathcal{T}=\langle \mathcal{Y}, \mathcal{X}\rangle$ is a semi-orthogonal decomposition if and only if the pair $(\mathcal{X}, \mathcal{Y})$ is a stable $t$-structure of $\mathcal{T}$; see \cite[Definition~9.14]{Miy}.

The semi-orthogonal decomposition  $\mathcal{T}=\langle \mathcal{Y}, \mathcal{X}\rangle$  is said to be \emph{orthogonal}, if in addition ${\rm Hom}_\mathcal{T}(\mathcal{Y}, \mathcal{X})=0$. In this situation, any object $A$ in $\mathcal{T}$ is uniquely decomposed as $A\simeq X\oplus Y$ with $X\in \mathcal{X}$ and $\mathcal{Y}$;  moreover, $\mathcal{T}=\langle \mathcal{X}, \mathcal{Y}\rangle$ is also a semi-orthogonal decomposition.

We relate pre-weight structures to semi-orthogonal decompositions.

\begin{lem}\label{lem:pws-fd}
Let $(\mathcal{U}_{\geq 0}, \mathcal{U}_{\leq 0})$  be a pre-weight structure on $\mathcal{T}$ with core $\mathcal{C}$ and dimension $d$. Then the following statements are equivalent:
\begin{enumerate}
\item $d=-1$;
\item $\mathcal{C}=0$;
 \item $\mathcal{T}=\langle \mathcal{U}_{\leq 0}, \mathcal{U}_{\geq 0}\rangle$ is a semi-orthogonal decomposition.
\end{enumerate}
\end{lem}

\begin{proof}
For ``(1) $\Rightarrow$ (2)", we assume that $d=-1$, or equivalently, ${\rm Hom}_\mathcal{T}(\mathcal{U}_{\geq 0}, \mathcal{U}_{\leq 0})=0$. In particular, ${\rm Hom}_\mathcal{T}(\mathcal{C}, \mathcal{C})=0$, which  implies that $\mathcal{C}=0$.

 For ``(2) $\Rightarrow$ (3)", we apply Lemma~\ref{lem:pws}(1) and infer that $\mathcal{U}_{\geq 0}=\mathcal{U}_{>0}$. Since $\mathcal{U}_{>0}=\Sigma^{-1}(\mathcal{U}_{\geq 0})$, we infer that $\mathcal{U}_{\geq 0}$ is a triangulated subcategory of $\mathcal{T}$. Similarly, $\mathcal{U}_{\leq 0}$ is also a triangulated subcategory. In particular, Definition~\ref{defn:pws}(1) might be rewritten as $\mathcal{U}_{\geq 0}\ast \mathcal{U}_{\leq 0}=\mathcal{T}$. By Lemma~\ref{lem:pws}(5), we have $\mathcal{U}_b=0$. Then Lemma~\ref{lem:factor} implies that ${\rm Hom}_\mathcal{T} (\mathcal{U}_{\geq 0},\mathcal{U}_{\leq 0})=0$. Consequently, we have the required semi-orthogonal decomposition  $\mathcal{T}=\langle \mathcal{U}_{\leq 0}, \mathcal{U}_{\geq 0}\rangle$.

 The implication ``(3) $\Rightarrow$ (1)" is trivial.
\end{proof}

The following semi-orthogonal decomposition might follow from \cite[Theorem~B]{JK}. It might be viewed as an extension of the one in Lemma~\ref{lem:pws-fd}(3), in which case we have $\mathcal{U}_b=0$, $\mathcal{U}_+=\mathcal{U}_{\geq 0}$ and $\mathcal{U}_{-}=\mathcal{U}_{\leq 0}$.

\begin{prop}\label{prop:semi-ortho}
Let $(\mathcal{U}_{\geq 0}, \mathcal{U}_{\leq 0})$  be a pre-weight structure on $\mathcal{T}$. Then we have a semi-orthogonal decomposition $\mathcal{T}/{\mathcal{U}_b}=\langle\mathcal{U}_-/{\mathcal{U}_b}, \mathcal{U}_{+}/{\mathcal{U}_b} \rangle$. Consequently, the following two canonical functors
$$\mathcal{U}_+/{\mathcal{U}_b}\stackrel{\sim}\longrightarrow \mathcal{T}/{\mathcal{U}_{-}} \mbox{ and } \mathcal{U}_{-}/{\mathcal{U}_b}\stackrel{\sim}\longrightarrow \mathcal{T}/{\mathcal{U}_{+}}$$
are triangle equivalences.
\end{prop}

\begin{proof}
Recall that both $\mathcal{U}_-$ and $\mathcal{U}_+$ are triangulated subcategories of $\mathcal{T}$. Therefore, both  $\mathcal{U}_-/{\mathcal{U}_b}$ and $\mathcal{U}_+/{\mathcal{U}_b}$ are triangulated subcategories of $\mathcal{T}/\mathcal{U}_b$. By Definition~\ref{defn:pws}(1), we infer that $\mathcal{U}_+\ast \mathcal{U}_{-}=\mathcal{T}$. This implies that
$$(\mathcal{U}_+/{\mathcal{U}_b})\ast (\mathcal{U}_+/{\mathcal{U}_b})=\mathcal{T}/{\mathcal{U}_b}.$$
For the required semi-orthogonal decomposition, it remains to prove that
$${\rm Hom}_{\mathcal{T}/{\mathcal{U}_b}}(X, Y)=0$$
 for any $X\in \mathcal{U}_+$ and $Y\in \mathcal{U}_{-}$. Indeed, any morphism $\theta\in {\rm Hom}_{\mathcal{T}/{\mathcal{U}_b}}(X, Y)$ is represented by a roof
 $$X\stackrel{f}\longrightarrow Y'\stackrel{s}\longleftarrow Y,$$ where the cone of $s$ lies in $\mathcal{U}_b$. It follows that $Y'$ belongs to $\mathcal{U}_{-}$. By Lemma~\ref{lem:factor}, the morphism $f$ factors through some object in $\mathcal{U}_+\cap \mathcal{U}_{-}=\mathcal{U}_b$. It follows that $\theta=0$ in $\mathcal{T}/{\mathcal{U}_b}$, as required.

We apply \cite[Proposition~1.6]{BK} to the obtained semi-orthogonal decomposition, and infer the equivalence on the left in the following identity.
$$\mathcal{U}_+/{\mathcal{U}_b}\simeq ({\mathcal{T}/{\mathcal{U}_{b}}})/({{\mathcal{U}_{-}/{\mathcal{U}_{b}}}}) \simeq  \mathcal{T}/{\mathcal{U}_{-}}.$$
The  equivalence on the right above is standard. This proves the equivalence $\mathcal{U}_+/{\mathcal{U}_b}\simeq  \mathcal{T}/{\mathcal{U}_{-}}$. The remaining one is proved similarly.
\end{proof}

\section{Examples}\label{sec:exm}

In this section, we study the canonical pre-weight structures on the homotopy category and  derived category of unbounded (cochain) complexes.

Let $\mathcal{A}$ be an additive category. A complex in $\mathcal{A}$ is usually denoted by $X=(X^n, d_X^n)_{n\in \mathbb{Z}}$, where the differentials $d_X^n\colon X^n\rightarrow X^{n+1}$ satisfy $d_X^{n+1}\circ d_X^n=0$. Denote by $\mathbf{K}(\mathcal{A})$ the  homotopy category of complexes in $\mathcal{A}$.

 For a complex $X$ and $n\in \mathbb{Z}$, we have the following brutal truncations:
 $$\sigma_{\geq n}(X)= \cdots \rightarrow 0\rightarrow X^n \stackrel{d_X^n}\rightarrow X^{n+1}\stackrel{d_X^{n+1}}\rightarrow X^{n+1}\rightarrow \cdots$$
 and
 $$\sigma_{<n}(X)=\cdots \rightarrow X^{n-3}\stackrel{d_X^{n-3}}\rightarrow X^{n-2}\stackrel{d_X^{n-2}}\rightarrow X^{n-1}\rightarrow 0\rightarrow \cdots.$$
 We have a canonical exact triangle in $\mathbf{K}(\mathcal{A})$:
 \begin{align}\label{tri:can}
 \sigma_{\geq n}(X)\xrightarrow{{\rm inc}} X\xrightarrow{{\rm pr}}  \sigma_{<n}(X)\longrightarrow \Sigma \sigma_{\geq n}(X).
 \end{align}
 Here, ``${\rm inc}$" and ``${\rm pr}$" denote the corresponding inclusion and projection, respectively.

 A complex $X$ is bounded-above (\emph{resp}., bounded-below) if $X^n=0$ for $n$ (\emph{resp}., $-n$) sufficiently large. A complex is bounded if it is both bounded-above and bounded-below. Denoted by $\mathbf{K}^{-}(\mathcal{A})$, $\mathbf{K}^{+}(\mathcal{A})$  and $\mathbf{K}^b(\mathcal{A})$ the full subcategories of $\mathbf{K}(\mathcal{A})$ formed by bounded-above, bounded-below and bounded complexes, respectively.

 The following terminology is convenient: a complex $X$ is non-negative (\emph{resp}. non-positive) if $X^n=0$ for $n<0$ (\emph{resp}. $n>0$).

We denote by $\mathcal{A}^\natural$ the \emph{idempotent completion} of $\mathcal{A}$. Its objects are pairs $(X, e)$ with $X\in \mathcal{A}$ and $e\colon X\rightarrow X$ satisfying $e^2=e$. A morphism $f\colon (X, e)\rightarrow (X', e')$ is given by a morphism $f\colon X\rightarrow X'$ satisfying $f=e'\circ f\circ e$. The composition of morphisms in $\mathcal{A}^\natural$ is induced by the one in $\mathcal{A}$. There is a fully faithful functor $\iota\colon \mathcal{A}\rightarrow \mathcal{A}^\natural$ sending $X$ to $(X, {\rm Id}_X)$, which is dense if and only if $\mathcal{A}$ is idempotent-complete.

Moreover, the functor $\iota$ induces a triangle equivalence
\begin{align}\label{equ:iota}
\mathbf{K}(\iota)\colon \mathbf{K}(\mathcal{A})\longrightarrow \mathbf{K}(\mathcal{A}^\natural),
\end{align}
which restricts to equivalences $\mathbf{K}^{-}(\mathcal{A})\simeq \mathbf{K}^{-}(\mathcal{A}^\natural)$ and $\mathbf{K}^{+}(\mathcal{A})\simeq \mathbf{K}^{+}(\mathcal{A}^\natural)$. For example, the object $(X, e)\in \mathcal{A}^\natural$, viewed as a stalk complex concentrated in degree zero, is identified with the following non-positive complex with the rightmost $X$ in degree zero
\begin{align}\label{equ:non-pos}
\cdots \longrightarrow X \stackrel{{\rm Id}_X-e} \longrightarrow X\stackrel{e} \longrightarrow X\stackrel{{\rm Id}_X-e}\longrightarrow X\longrightarrow 0\longrightarrow \cdots
\end{align}
in $\mathbf{K}(\mathcal{A})$. We mention that this complex is also isomorphic to the following non-negative complex
\begin{align}\label{equ:non-neg}
\cdots \longrightarrow 0 \longrightarrow X \stackrel{{\rm Id}_X-e} \longrightarrow X\stackrel{e} \longrightarrow X\stackrel{{\rm Id}_X-e}\longrightarrow X\longrightarrow \cdots
\end{align}
in $\mathbf{K}(\mathcal{A})$.

\begin{exm}\label{exm:K}
{\rm Let $\mathcal{A}$ be an additive category. Denote by $\mathcal{K}_{\geq 0}$ (\emph{resp}. $\mathcal{K}_{\leq 0}$) the full subcategory of $\mathbf{K}(\mathcal{A})$ formed by those complexes that are isomorphic to  some non-negative (\emph{resp.} non-positive) complexes.  In view of (\ref{tri:can}), we infer that $(\mathcal{K}_{\geq 0},\mathcal{K}_{\leq 0})$ is a pre-weight structure of dimension $0$, called the \emph{canonical pre-weight structure} on $\mathbf{K}(\mathcal{A})$. Indeed, it is a weight structure.

It is clear that $\mathcal{A}$ is contained in the core $\mathcal{K}_{\geq 0}\cap \mathcal{K}_{\leq 0}$. On the other hand, any object $X\in \mathcal{K}_{\geq 0}\cap \mathcal{K}_{\leq 0}$ is isomorphic to a direct summand of some object in $\mathcal{A}$. Indeed, we might assume that $X$ is non-negative. Consider the canonical triangle
$$ \sigma_{\geq 1}(X)\xrightarrow{{\rm inc}} X\xrightarrow{{\rm pr}}  X^0 \longrightarrow \Sigma \sigma_{\geq 1}(X).$$
Since $X$ belongs to $\mathcal{K}_{\leq 0}$, we infer that ``${\rm inc}$" is zero. Therefore, the morphism ``${\rm pr}$" is a split monomorphism. In view of the equivalence (\ref{equ:iota}), we infer an equivalence
\begin{align}\label{equ:core-K}
\mathcal{A}^\natural\simeq \mathcal{K}_{\geq 0}\cap \mathcal{K}_{\leq 0}
\end{align}
sending $(X, e)$ to the complex in (\ref{equ:non-pos}), which is also isomorphic to the one in (\ref{equ:non-neg}).

With respect to the pre-weight structure $(\mathcal{K}_{\geq 0},\mathcal{K}_{\leq 0})$ , we identify $\mathcal{K}_+$ with $\mathbf{K}^+(\mathcal{A})$, and $\mathcal{K}_-$ with $\mathbf{K}^-(\mathcal{A})$. We have $\mathbf{K}^b(\mathcal{A})\subseteq \mathcal{K}_b$. Indeed, in view of (\ref{equ:iota}) and (\ref{equ:core-K}), we have an equivalence
$$\mathcal{K}_b \simeq \mathbf{K}^b(\mathcal{A}^\natural),$$
which is further equivalent to $\mathbf{K}^b(\mathcal{A})^\natural$, the idempotent completion of $\mathbf{K}^b(\mathcal{A})$; see \cite[Proposition~3.4]{BN}. In particular, we can identify the two Verdier quotient categories: $$\mathbf{K}(\mathcal{A})/{\mathcal{K}_b}=\mathbf{K}(\mathcal{A})/{\mathbf{K}^b(\mathcal{A})}.$$ Consequently, Proposition~\ref{prop:semi-ortho} allows us to recover the semi-orthogonal decomposition
$$\mathbf{K}(\mathcal{A})/{\mathbf{K}^b(\mathcal{A})}=\langle \mathbf{K}^{-}(\mathcal{A})/{\mathbf{K}^b(\mathcal{A})}, \mathbf{K}^{+}(\mathcal{A})/{\mathbf{K}^b(\mathcal{A})}\rangle$$
in \cite{Chen-Wang}.
}
\end{exm}

We assume that the additive category $\mathcal{A}$ has arbitrary coproducts and products. Consequently, $\mathbf{K}(\mathcal{A})$ is bi-complete, whose coproducts and products are componentwise.

Assume that $\mathcal{T}$ is a bi-complete triangulated category. A triangulated subcategory of $\mathcal{T}$ is called \emph{localizing} (resp., \emph{colocalizing}) if it is closed under arbitrary coproducts (resp., arbitrary products); see \cite[Section~1]{BN}. For a class $\mathcal{S}$ of objects, we denote by ${\rm Loc}\langle \mathcal{S} \rangle$ (resp., ${\rm coLoc}\langle \mathcal{S}\rangle$) the smallest localizing (resp., colocalizing) subcategory of $\mathcal{T}$ containing $\mathcal{S}$.

The following result is standard.

\begin{lem}\label{lem:can-pws-K}
Assume that the additive category $\mathcal{A}$ has arbitrary coproducts and products. Consider the canonical pre-weight structure $(\mathcal{K}_{\geq 0}, \mathcal{K}_{\leq 0})$ on $\mathbf{K}(\mathcal{A})$. Then the following statements hold.
\begin{enumerate}
\item The pre-weight structure $(\mathcal{K}_{\geq 0}, \mathcal{K}_{\leq 0})$ is two-sided convergent.
\item $\mathcal{K}_{\leq 0}\subseteq {\rm Loc}\langle \mathcal{K}_b \rangle$ and $\mathcal{K}_{\geq 0}\subseteq {\rm coLoc}\langle \mathcal{K}_b \rangle$.
\end{enumerate}
\end{lem}

\begin{proof}
Take any object $X\in \mathcal{K}_{\leq 0}$. We may assume that $X$ is non-positive. The inclusions $\sigma_{\geq -n}(X)\rightarrow X$  and ${\rm inc}_n\colon \sigma_{\geq -n}(X)\rightarrow \sigma_{\geq -n-1}(X)$ form a left tower of $X$. Moreover, we have  an isomorphism
$${\rm colim}_{n\geq 0}\;\sigma_{\geq -n}(X)\simeq X.$$
The following componentwise-split short exact sequence of complexes
$$ 0\longrightarrow \prod_{n\geq 0}\sigma_{\geq -n}(X) \stackrel{1-{\rm inc}}\longrightarrow \prod_{n\geq 0}\sigma_{\geq -n}(X)\longrightarrow X\longrightarrow 0$$
allows us to identify $X$ with  the homotopy colimit ${\rm hocolim}\; \sigma_{\geq -n}(X)$. Then we infer that the left tower above of $X$ is convergent and $X\in {\rm Loc}\langle \mathcal{K}_b\rangle$. Therefore, the pre-weight structure is left convergent and $\mathcal{K}_{\leq 0}\subseteq {\rm Loc}\langle \mathcal{K}_b \rangle$. The remaining statements are proved dually.
\end{proof}

In what follows, we assume that $\mathcal{A}$ is an abelian category. Denote by $\mathbf{D}(\mathcal{A})$ the derived category of complexes in $\mathcal{A}$. By identifying any object in $\mathcal{A}$ with the corresponding stalk complex concentrated in degree zero, we always view $\mathcal{A}$ as a full subcategory of $\mathbf{D}(\mathcal{A})$.

Denote by $\mathbf{D}^{-}(\mathcal{A})$, $\mathbf{D}^{+}(\mathcal{A})$ and $\mathbf{D}^{b}(\mathcal{A})$ the derived categories of bounded-above, bounded-below and bounded complexes, respectively. All these categories are naturally viewed as triangulated subcategories of $\mathbf{D}(\mathcal{A})$.

\begin{exm}\label{exm:D}
{\rm
Denote by $\mathcal{D}_{\geq 0}$ (\emph{resp}. $\mathcal{D}_{\leq 0}$) the full subcategory of $\mathbf{D}(\mathcal{A})$ consisting of complexes that are isomorphic to non-negative complexes (\emph{resp}. non-positive complexes). Using the same exact triangles (\ref{tri:can}), we infer that $(\mathcal{D}_{\geq 0}, \mathcal{D}_{\leq 0})$ is a pre-weight structure on $\mathbf{D}(\mathcal{A})$, called the \emph{canonical pre-weight structure} on $\mathbf{D}(\mathcal{A})$.

The core of $(\mathcal{D}_{\geq 0}, \mathcal{D}_{\leq 0})$  is naturally identified with $\mathcal{A}$. Moreover, we identify $\mathcal{D}_{-}$ (\emph{resp}. $\mathcal{D}_{+}$ and  $\mathcal{D}_b$ ) with $\mathbf{D}^{-}(\mathcal{A})$ (\emph{resp}. $\mathbf{D}^{+}(\mathcal{A})$  and $\mathbf{D}^b(\mathcal{A})$).

We apply Proposition~\ref{prop:semi-ortho} to  $(\mathcal{D}_{\geq 0}, \mathcal{D}_{\leq 0})$, and obtain the following semi-orthogonal decomposition
$$\mathbf{D}(\mathcal{A})/{\mathbf{D}^b(\mathcal{A})}=\langle \mathbf{D}^-(\mathcal{A})/{\mathbf{D}^b(\mathcal{A})}, \mathbf{D}^+(\mathcal{A})/{\mathbf{D}^b(\mathcal{A})}\rangle.$$
We mention that this decomposition is orthogonal. Indeed, it is well known that $(\mathcal{D}_{\leq 0}, \mathcal{D}_{\geq 0})$ is a t-structure on $\mathbf{D}(\mathcal{A})$; see  \cite[Section~1.3]{BBD}. In particular, we have $\mathbf{D}^-(\mathcal{A})\ast \mathbf{D}^+(\mathcal{A})=\mathbf{D}(\mathcal{A})$. Applying \cite[Theorem~B]{JK}, we have another semi-orthogonal decomposition
$$\mathbf{D}(\mathcal{A})/{\mathbf{D}^b(\mathcal{A})}=\langle \mathbf{D}^+(\mathcal{A})/{\mathbf{D}^b(\mathcal{A})}, \mathbf{D}^-(\mathcal{A})/{\mathbf{D}^b(\mathcal{A})}\rangle.$$
This implies that both decompositions above are orthogonal.
}
\end{exm}

Recall that the \emph{global dimension} of $\mathcal{A}$, denoted by ${\rm gl.dim}(\mathcal{A})$, is defined to the supremum of $\{n\geq 0\; |\; {\rm Ext}_\mathcal{A}^n(-, -) \neq 0\}$.

\begin{lem}\label{lem:can-pws-D-fd}
The dimension of the canonical pre-weight structure $(\mathcal{D}_{\geq 0}, \mathcal{D}_{\leq 0})$ on $\mathbf{D}(\mathcal{A})$ is at least ${\rm gl.dim}(\mathcal{A})$. If $\mathcal{A}$ has enough projectives or injectives, the two dimensions equal.
\end{lem}

\begin{proof}
We observe that if ${\rm Ext}_\mathcal{A}^n(-, -) \neq 0$, we have ${\rm Hom}_{\mathbf{D}(\mathcal{A})}(\mathcal{D}_{\geq 0}, \mathcal{D}_{\leq -n})\neq 0$. This observation implies the first statement. The second one follows immediately from \cite[Proposition~3.6(2)]{CLZ}.
\end{proof}

We assume that the abelian category $\mathcal{A}$ has arbitrary coproducts and products such that coproducts and products of any exact sequence are exact. By \cite[Example~1.6]{BN}, $\mathbf{D}(\mathcal{A})$ is bi-complete, whose coproducts and products are componentwise.

The same proof with Lemma~\ref{lem:can-pws-K} yields the following result.

\begin{lem}\label{lem:can-pws-D}
Assume that the abelian category $\mathcal{A}$ has arbitrary coproducts and products such that coproducts and products of any exact sequence are exact. Consider the canonical pre-weight structure $(\mathcal{D}_{\geq 0}, \mathcal{D}_{\leq 0})$ on $\mathbf{D}(\mathcal{A})$. Then the following statements hold.
\begin{enumerate}
\item The pre-weight structure $(\mathcal{D}_{\geq 0}, \mathcal{D}_{\leq 0})$ is two-sided convergent.
\item $\mathcal{D}_{\leq 0}\subseteq {\rm Loc}\langle \mathcal{D}_b \rangle$ and $\mathcal{D}_{\geq 0}\subseteq {\rm coLoc}\langle \mathcal{D}_b \rangle$. \hfill $\square$
\end{enumerate}
\end{lem}

\section{The action on objects and fully-faithfulness}\label{sec:triv-act}

In this section, we fix a triangulated category $\mathcal{T}$, which is endowed with a pre-weight structure $(\mathcal{U}_{\geq 0}, \mathcal{U}_{\leq 0})$. Denote by $\mathcal{C}=\mathcal{U}_{\geq 0}\cap \mathcal{U}_{\leq 0}$ its core and by $\mathcal{U}_b$ its bounded part.

We prove that a triangle endofunctor on $\mathcal{T}$ acts trivially on objects, provided that its restriction to the core is isomorphic to the identity functor; see Theorem~\ref{thm:trivial-action}. We prove that  a certain triangle functor starting from $\mathcal{T}$ is fully faithful, provided that so is its restriction to $\mathcal{U}_b$; see Theorem~\ref{thm:extend-ff}. These two theorems will be crucial in the criterion on when an endofunctor is isomorphic to a pseudo-identity; see Section~\ref{sec:ps}.

Let $\mathcal{T}'$ be another triangulated category. Recall that a \emph{triangle functor} $(F, \omega)\colon \mathcal{T}\rightarrow \mathcal{T}'$ consists of an additive functor $F\colon \mathcal{T}\rightarrow \mathcal{T}'$ and a natural isomorphism $\omega\colon F\Sigma\rightarrow \Sigma F$, called the \emph{connecting isomorphism},  such that any exact triangle $X\stackrel{a}\rightarrow Y\stackrel{b}\rightarrow Z\stackrel{c}\rightarrow \Sigma(X)$ in $\mathcal{T}$ is sent to an exact triangle
$F(X)\stackrel{F(a)}\rightarrow F(Y)\stackrel{F(b)}\rightarrow F(Z)\xrightarrow{\omega_X\circ F(c)} \Sigma F(X)$ in $\mathcal{T}'$. For each $n\geq 1$, we define a natural isomorphism $\omega^n\colon F\Sigma^n\rightarrow \Sigma^nF$ inductively such that $\omega^1=\omega$ and $\omega^{n+1}=\Sigma(\omega^n)\circ \omega{\Sigma^n}$. We set $\omega^0={\rm Id}_F$.

When the natural isomorphism $\omega$ is understood or not needed, we also call $F$ a triangle functor. For example, we always understand the identity triangle endofunctor ${\rm Id}_\mathcal{T}$ as the pair $({\rm Id}_\mathcal{T}, {\rm Id}_\Sigma)$.

\subsection{Trivial actions on objects}

In this subsection, we fix a triangle endofunctor  $(F, \omega)\colon \mathcal{T}\rightarrow \mathcal{T}$ such that for each object $A\in \mathcal{C}$, there is an isomorphism $\eta_A\colon F(A)\rightarrow A$; moreover, $\eta_A$ is natural in $A\in \mathcal{C}$.

The following result is inspired by \cite[Proposition~7.1]{Ric89} and \cite[Proposition~3.2]{CY}.

\begin{prop}\label{prop:object-b}
 Assume that ${\rm Hom}_\mathcal{T}(\mathcal{C}, \Sigma^{-n}(\mathcal{C}))=0$ for any $n\geq 1$. Then for each object $X\in \mathcal{U}_b$, $F(X)$ is isomorphic to $X$.
\end{prop}

\begin{proof}
Applying the suspension functor, it suffices prove the statement for $X\in \mathcal{U}_{[-n, 0]}$ with $n\geq 0$. By Corollary~\ref{cor:ast}, there exist a sequence of morphisms
$$X_0\stackrel{\phi_0}\longrightarrow X_1 \longrightarrow \cdots \longrightarrow  X_{n-1}\stackrel{\phi_{n-1}}\longrightarrow X_{n}=X$$
such that each $X_i\in \mathcal{U}_{[-i, 0]}$ and that the cone of $\phi_i$ lies in $\Sigma^{i+1}(\mathcal{C})$. Therefore, we can fix an exact triangle
$$X_i\stackrel{\phi_i}\longrightarrow X_{i+1}\stackrel{\pi_{i+1}} \longrightarrow \Sigma^{i+1}(A_{i})\stackrel{c_i}\longrightarrow \Sigma(X_i)$$
for each $0\leq i\leq n-1$, with each $A_{i}\in \mathcal{C}$. We set $A_{-1}=X_0$ and $\pi_{-1}={\rm Id}_{X_0}$.

We claim that for each $0\leq i\leq n$, there exists an isomorphism $\theta_i\colon F(X_i)\rightarrow X_i$ such that $\pi_i\circ \theta_i=\Sigma^i(\eta_{A_{i-1}})\circ \omega^i_{A_{i-1}}\circ F(\pi_i)$.

We prove the claim by induction. Set $\theta_0=\eta_{X_0}$. We assume that $\theta_{i}$ is already established. Consider the following diagram with rows being exact triangles.
\[
\xymatrix{
F(X_i)\ar[d]_-{\theta_i} \ar[r]^-{F(\phi_i)} & F(X_{i+1}) \ar[r]^-{F(\pi_{i+1})} & F\Sigma^{i+1}(A_{i}) \ar[d]_-{\Sigma^{i+1}(\eta_{A_i}) \circ \omega^{i+1}_{A_{i}}} \ar[rr]^-{\omega_{X_i}\circ F(c_i)} && \Sigma F(X_i) \ar[d]^-{\Sigma(\theta_i)}\\
X_i\ar[r]^-{\phi_i} & X_{i+1} \ar[r]^-{\pi_{i+1}} & \Sigma^{i+1}(A_{i}) \ar[rr]^-{c_i} && \Sigma(X_i)
}\]
It suffices to prove that the square on the right hand side commutes, and then by (TR3) we obtain the required isomorphism $\theta_{i+1}$.

For this end, we apply Corollary~\ref{cor:vanishing} to obtain the following identity.
$$0={\rm Hom}_\mathcal{T}(A_i, \Sigma^{-i-1}(X_i))\simeq {\rm Hom}_\mathcal{T}(F\Sigma^{i+1}(A_i), X_i)$$
Therefore, the following map induced by $\Sigma(\pi_i)\colon \Sigma(X_i)\rightarrow \Sigma^{i+1}(A_{i-1})$
$${\rm Hom}_\mathcal{T}(F\Sigma^{i+1}(A_i), \Sigma(X_i))\longrightarrow {\rm Hom}_\mathcal{T}(F\Sigma^{i+1}(A_i), \Sigma^{i+1}(A_{i-1}))$$
is injective. Consequently, the required commutativity above follows from the identity below.
\begin{align}\label{equ:comm-diag-1}
\Sigma(\pi_i)\circ \Sigma(\theta_i)\circ \omega_{X_i}\circ F(c_i)=\Sigma(\pi_i)\circ c_i\circ \Sigma^{i+1}(\eta_{A_i})\circ \omega^{i+1}_{A_i}.
\end{align}

To prove (\ref{equ:comm-diag-1}), we fix a unique morphism $u\colon A_i\rightarrow A_{i-1}$ in $\mathcal{C}$ satisfying
$$\Sigma^{i+1}(u)=\Sigma(\pi_i)\circ c_i.$$
 By the inductive hypothesis, we have the first equality of the following identity.
\begin{align*}
\Sigma(\pi_i)\circ \Sigma(\theta_i)\circ \omega_{X_i}\circ F(c_i) &= \Sigma^{i+1}(\eta_{A_{i-1}})\circ \Sigma(\omega^i_{A_{i-1}}) \circ \Sigma F(\pi_i)\circ \omega_{X_i}\circ F(c_i)\\
&=\Sigma^{i+1}(\eta_{A_{i-1}})\circ \Sigma(\omega^i_{A_{i-1}}) \circ \omega_{\Sigma^{i+1}(A_{i-1})}\circ F\Sigma(\pi_i)\circ F(c_i)\\
&=\Sigma^{i+1}(\eta_{A_{i-1}})\circ \omega^{i+1}_{A_{i-1}}\circ F\Sigma(\pi_i)\circ F(c_i)\\
&=\Sigma^{i+1}(\eta_{A_{i-1}})\circ \omega^{i+1}_{A_{i-1}}\circ F\Sigma^{i+1}(u)\\
&=\Sigma^{i+1}(\eta_{A_{i-1}})\circ \Sigma^{i+1}F(u)\circ \omega^{i+1}_{A_i}\\
&=\Sigma^{i+1}(u)\circ \Sigma^{i+1}(\eta_{A_i})\circ \omega^{i+1}_{A_i}\\
&= \Sigma(\pi_i)\circ c_i \circ \Sigma^{i+1}(\eta_{A_i})\circ \omega^{i+1}_{A_i}
\end{align*}
Here, the second (\emph{resp}. fifth, sixth) equality uses the naturalness of $\omega$ (\emph{resp}. $\omega^{i+1}$, $\eta$), the third equality uses the definition of $\omega^{i+1}$, the fourth and last equalities use the definition of $u$. This proves (\ref{equ:comm-diag-1}), which implies the required commutativity above.
\end{proof}

\begin{prop}
 Suppose that $\mathcal{T}$ is cocomplete and that $F$ preserves  arbitrary  coproducts.  Assume that the pre-weight structure $(\mathcal{U}_{\geq 0}, \mathcal{U}_{\leq 0})$ is left convergent and that ${\rm Hom}_\mathcal{T}(\mathcal{C}, \Sigma^{-n}(\mathcal{C}))=0$ for any $n\geq 1$. Then for each object $X\in \mathcal{U}_{-}$, $F(X)$ is isomorphic to $X$.
\end{prop}

\begin{proof}
Applying the suspension functor, it suffices to prove the result for  $X\in \mathcal{U}_{\leq 0}$. Take a left tower $(\iota_n, \phi_n)_{n\geq 0}$ of $X$ which is convergent. In particular, we identify $X$ with ${\rm hocolim}\; X_n$. By the proof of Proposition~\ref{prop:object-b}, there are isomorphisms $\theta_n\colon F(X_n)\rightarrow X_n$. Moreover, by  the inductive construction in the proof, we observe that these isomorphisms do satisfy the following condition.
$$\phi_n\circ \theta_n=\theta_{n+1}\circ F(\phi_n)$$
Therefore, the leftmost square in the following diagram commutes.
\[\xymatrix{
F(\coprod_{n\geq 0} X_n) \ar[d]_-{\coprod_{n\geq 0}\theta_n} \ar[r]^-{F(1-\phi)} & F(\coprod_{n\geq 0} X_n) \ar[d]^-{\coprod_{n\geq 0}\theta_n}\ar[r]^-{F(\gamma)}  & F({\rm hocolim}\; X_n) \ar[r] & \Sigma F(\coprod_{n\geq 0} X_n) \ar[d]^-{\Sigma(\coprod_{n\geq 0}\theta_n)}\\
\coprod_{n\geq 0} X_n \ar[r]^{1-\phi} & \coprod_{n\geq 0} X_n \ar[r]^-\gamma & {\rm hocolim}\; X_n \ar[r] & \Sigma (\coprod_{n\geq 0} X_n)
}\]
Here, when using $\coprod_{n\geq 0}\theta_n$, we  implicitly identify $F(\coprod_{n\geq 0} X_n)$ with $\coprod_{n\geq 0} F(X_n)$.  By (TR3), we have an isomorphism $F({\rm hocolim}\; X_n)\rightarrow {\rm hocolim}\; X_n$ making the diagram commute. It follows that $F(X)\simeq X$.
\end{proof}

\begin{rem}\label{rem:X}
Assume that $X\in \mathcal{U}_{\leq 0}$. Recall from Definition~\ref{defn:left-conv} that the isomorphism $\kappa\colon {\rm hocolim}\; X_n\rightarrow X$ is required to satisfy  $\kappa\circ \gamma_n=\iota_n$ for each $n\geq 0$. It follows from $\kappa\circ \gamma_0=\iota_0$ and $\theta_0=\eta_{X_0}$ that the isomorphism $\theta_X\colon F(X)\rightarrow  X$ above satisfies
\begin{align}\label{equ:X}
\theta_X \circ F(\iota_0)=\iota_0\circ \eta_{X_0}.
\end{align}
Here, we mention that the object $X_0$ belongs to the core $\mathcal{C}$ and thus $\eta_{X_0}$ makes sense.
\end{rem}

By duality, we have the following result.

\begin{prop}
 Suppose that $\mathcal{T}$ is complete and that $F$ preserves  arbitrary  products.  Assume  that ${\rm Hom}_\mathcal{T}(\mathcal{C}, \Sigma^{-n}(\mathcal{C}))=0$ for any $n\geq 1$ and that the pre-weight structure $(\mathcal{U}_{\geq 0}, \mathcal{U}_{\leq 0})$ is right convergent. Then for each object $Y\in \mathcal{U}_{+}$, $F(Y)$ is isomorphic to $Y$. \hfill $\square$
\end{prop}

\begin{rem}
Assume that $Y\in\mathcal{U}_{\geq 0}$  has  a convergent right tower $(\pi_n, \psi_n)_{n\geq 0}$. In particular, we have $\pi_0\colon Y\rightarrow Y_0$ with $Y_0\in \mathcal{C}$.  Recall from  Definition~\ref{defn:right-conv} that the isomorphism $\lambda\colon Y\rightarrow {\rm holim} Y_n$ is required to satisfy $\sigma_0\circ \lambda=\pi_0$.  Then we infer that  the isomorphism  $\tau_Y\colon F(Y)\rightarrow Y$ above necessarily satisfies
\begin{align}\label{equ:Y}
\pi_0\circ \tau_Y=\eta_{Y_0}\circ F(\pi_0).
\end{align}
\end{rem}

The following result might be viewed as  an infinite version of Proposition~\ref{prop:object-b}. A triangle functor is said to be \emph{bi-continuous} if it preserves  arbitrary  coproducts and products.

\begin{thm}\label{thm:trivial-action}
 Suppose that $\mathcal{T}$ is bi-complete and that $F$ is bi-continuous.  Assume that the pre-weight structure $(\mathcal{U}_{\geq 0}, \mathcal{U}_{\leq 0})$ is  two-sided convergent and that ${\rm Hom}_\mathcal{T}(\mathcal{C}, \Sigma^{-n}(\mathcal{C}))=0$ for any $n\geq 1$. Then for each object $A\in \mathcal{T}$, $F(A)$ is isomorphic to $A$.
\end{thm}

\begin{proof}
Applying $\Sigma^{-1}$ to Definition~\ref{defn:pws}(1), we have $\mathcal{U}_{\geq 1}\ast \mathcal{U}_{\leq 0}=\mathcal{T}$.  Therefore, we can fix an exact triangle
$$A\stackrel{a}\longrightarrow X\stackrel{f}\longrightarrow Y\stackrel{b}\longrightarrow \Sigma(A)$$
with $X\in \mathcal{U}_{\leq 0}$ and $Y\in \mathcal{U}_{\geq 0}$. Fix a left tower $(\iota_n, \phi_n)_{n\geq 0}$ for $X$ and a right tower $(\pi_n, \psi_n)_{n\geq 0}$ for $Y$, both of which are assumed to be convergent. By Remark~\ref{rem:X}, we have an isomorphism $\theta_X\colon F(X)\rightarrow X$ satisfying (\ref{equ:X}). Similarly, we have an isomorphism $\tau_Y\colon F(Y)\rightarrow Y$ satisfying (\ref{equ:Y}).

By (TR3), it suffices to prove the commutativity of the middle square below.
\[\xymatrix{
F(A) \ar[r]^-{F(a)} & F(X)\ar[d]_-{\theta_X} \ar[r]^-{F(f)} & F(Y)\ar[d]^-{\tau_Y} \ar[r]^-{\omega_A\circ F(b)} & \Sigma F(A)\\
A \ar[r]^-{a} & X \ar[r]^-{f} & Y \ar[r]^-{b} & \Sigma (A)
}\]
For this commutativity, we consider the following two exact triangles appearing in the two towers:
\begin{align}\label{tri:triv1}
X_0\stackrel{\iota_0} \longrightarrow X\longrightarrow X'\longrightarrow \Sigma(X_0)
\end{align}
and
\begin{align}\label{tri:triv2}
Y'\longrightarrow Y \stackrel{\pi_0}\longrightarrow Y_0\longrightarrow \Sigma(Y').
\end{align}
We observe that $X'\in \mathcal{U}_{\leq -1}$ and $Y'\in \mathcal{U}_{\geq 1}$; moreover, we have isomorphisms $F(X')\simeq X'$ and $F(X)\simeq X$ by Proposition~\ref{prop:object-b}. By Lemma~\ref{lem:vanishing}, we have
$${\rm Hom}_{\mathcal{T}}(F(X'), Y_0)\simeq {\rm Hom}_{\mathcal{T}}(X', Y_0)=0.$$
Applying ${\rm Hom}_\mathcal{T}(F(-), Y_0)$ to (\ref{tri:triv1}), we infer that  the following map induced by $F(\iota_0)$
$${\rm Hom}_{\mathcal{T}}(F(X), Y_0)\longrightarrow {\rm Hom}_{\mathcal{T}}(F(X_0), Y_0)$$
is injective. Similarly, by Proposition~\ref{prop:vanishing}, we have
$${\rm Hom}_{\mathcal{T}}(F(X), Y')\simeq {\rm Hom}_{\mathcal{T}}(X, Y')=0.$$
Applying ${\rm Hom}_\mathcal{T}(F(X), -)$ to (\ref{tri:triv2}), we infer that  the following map induced by $\pi_0$
$${\rm Hom}_{\mathcal{T}}(F(X), Y)\longrightarrow {\rm Hom}_{\mathcal{T}}(F(X), Y_0)$$
is also injective. By the two injective maps above, the required commutativity is equivalent to the following identity.
$$\pi_0\circ (f\circ \theta_X)\circ F(\iota_0)=\pi_0\circ (\tau_Y\circ F(f))\circ F(\iota_0)$$

Finally, we prove the identity above as follows:
\begin{align*}
\pi_0\circ f\circ \theta_X\circ F(\iota_0) & =(\pi_0\circ f \circ \iota_0)\circ \eta_{X_0}\\
                                            &= \eta_{Y_0}\circ F(\pi_0)\circ F(f)\circ F(\iota_0)\\
                                            &=\pi_0\circ \tau_Y \circ F(f)\circ F(\iota_0).
\end{align*}
Here, the first equality uses (\ref{equ:X}) and the third one uses (\ref{equ:Y}). For the second equality, we observe that $\pi_0\circ f \circ \iota_0\colon X_0\rightarrow Y_0$ is a morphism in $\mathcal{C}$ and then we apply the naturalness of $\eta$.
\end{proof}

\subsection{Extending the fully-faithfulness}

 Let $\mathcal{T}'$ be another triangulated category. A triangle functor $F\colon \mathcal{T}\rightarrow \mathcal{T}'$ is called \emph{bounded} with respect to the pre-weight structure $(\mathcal{U}_{\geq 0}, \mathcal{U}_{\leq 0})$ on $\mathcal{T}$, if there exists a natural number $n_0$ such that  ${\rm Hom}_{\mathcal{T}'}(F(\mathcal{U}_{\geq 0}), F(\mathcal{U}_{\leq -n_0}))=0$, or equivalently, ${\rm Hom}_{\mathcal{T}'}(F(X), \Sigma^{n_0}F(Y))=0$ for any $X\in \mathcal{U}_{\geq 0}$ and $Y\in \mathcal{U}_{\leq 0}$.

In this subsection, we assume that both $\mathcal{T}$ and $\mathcal{T}'$ are bi-complete.

\begin{thm}\label{thm:extend-ff}
Let $F\colon \mathcal{T}\rightarrow \mathcal{T}'$ be a bi-continuous triangle functor. Assume that the following conditions are fulfilled.
\begin{enumerate}
\item There is a pre-weight structure $(\mathcal{U}_{\geq 0}, \mathcal{U}_{\leq 0})$ on $\mathcal{T}$ with finite dimension.
\item We have $\mathcal{U}_{\leq 0}\subseteq {\rm Loc}\langle \mathcal{U}_b\rangle$ and $\mathcal{U}_{\geq 0}\subseteq {\rm coLoc}\langle \mathcal{U}_b\rangle$.
\item The functor $F$ is bounded with respect to  $(\mathcal{U}_{\geq 0}, \mathcal{U}_{\leq 0})$.
\item The restriction $F|_{\mathcal{U}_b}\colon \mathcal{U}_b\rightarrow \mathcal{T}'$ is fully faithful.
\end{enumerate}
Then the given functor $F$ is fully faithful.
\end{thm}

\begin{proof}
We will show that for any object $X, Y$ in $\mathcal{T}$, the map
$$F_{X, Y}\colon {\rm Hom}_\mathcal{T}(X, Y)\longrightarrow {\rm Hom}_{\mathcal{T}'}(F(X), F(Y)), \; f\mapsto F(f),$$
is an isomorphism.

\emph{Step 1.}\; Fix a natural number  $n_0$ such that ${\rm Hom}_{\mathcal{T}'}(F(\mathcal{U}_{\geq 0}), F(\mathcal{U}_{\leq -n_0}))=0$. Denote by $d$ the dimension of the pre-weight structure. Assume that $X\in \mathcal{U}_{\geq n}$ and $Y\in \mathcal{U}_{\leq m}$. Set
$$N={\rm max}\{n+d+2, m+d+2, n+n_0+2, m+n_0+2\}.$$
The following argument is similar to the one in the proof of Proposition~\ref{prop:pws-fd}.

By Lemma~\ref{lem:pws}(1), we have two exact triangles:
$$X'\longrightarrow X\stackrel{a}\longrightarrow X''\longrightarrow \Sigma(X') \mbox{ and } Y'\stackrel{b}\longrightarrow Y\longrightarrow Y''\longrightarrow \Sigma(Y'),$$
with $X'\in \mathcal{U}_{\geq N}$, $X''\in \mathcal{U}_{[n, N-1]}$, $Y'\in \mathcal{U}_{[-N+1, m]}$ and $Y''\in \mathcal{U}_{\leq -N}$. By the choice of $N$, we observe that
$${\rm Hom}_{\mathcal{T}'}(F(X'), F(Y))=0={\rm  Hom}_{\mathcal{T}'}(\Sigma F(X'), F(Y)).$$
 Applying ${\rm Hom}_{\mathcal{T}'}(-, F(Y))$ to the following exact triangle
$$F(X')\longrightarrow F(X)\stackrel{F(a)}\longrightarrow F(X'')\longrightarrow \Sigma F(X'),$$
we infer that $F(a)$ induces an isomorphism
$${\rm Hom}_{\mathcal{T}'}(F(X''), F(Y))\longrightarrow {\rm Hom}_{\mathcal{T}'}(F(X), F(Y)).$$
Similar, the morphism $F(b)$ induces the following isomorphism
$${\rm Hom}_{\mathcal{T}'}(F(X''), F(Y'))\longrightarrow {\rm Hom}_{\mathcal{T}'}(F(X''), F(Y)).$$
Combining the two isomorphisms, we infer that the lower row in the following commutative square is an isomorphism.
\[
\xymatrix{ {\rm Hom}_\mathcal{T}(X'', Y') \ar[d]_-{F_{X'', Y'}}\ar[rr]^{b\circ -\circ a} && {\rm Hom}_\mathcal{T}(X, Y) \ar[d]^-{F_{X, Y}}\\
{\rm Hom}_{\mathcal{T}'}(F(X''), F(Y'))\ar[rr]^{F(b)\circ -\circ F(a)} && {\rm Hom}_{\mathcal{T}'}(F(X), F(Y))
}\]
Since both $X''$ and $Y'$ belong to $\mathcal{U}_b$, the map $F_{X'', Y'}$ is an isomorphism by (4). The upper row is also an isomorphism by Proposition~\ref{prop:pws-fd}. It follows that so is the right vertical arrow. We conclude that $F_{X, Y}$ is an isomorphism for any $X\in \mathcal{U}_+$ and $Y\in \mathcal{U}_{-}$.

\emph{Step 2.}\; We claim that ${\rm Loc}\langle \mathcal{U}_{+}\rangle=\mathcal{T}={\rm coLoc}\langle \mathcal{U}_{-}\rangle$. By (2), we have $\mathcal{U}_{\leq 0}\subseteq {\rm Loc}\langle \mathcal{U}_{b}\rangle \subseteq {\rm Loc}\langle \mathcal{U}_{+}\rangle$. Therefore, we have $\mathcal{T}=\mathcal{U}_{\geq 1}\ast \mathcal{U}_{\leq 0}\subseteq {\rm Loc}\langle \mathcal{U}_{+}\rangle \ast {\rm Loc}\langle \mathcal{U}_{+}\rangle ={\rm Loc}\langle \mathcal{U}_{+}\rangle$. By duality, we have the remaining equality.

For a fixed object $Y\in \mathcal{U}_{-}$, we consider the following full subcategory of $\mathcal{T}$:
$$\mathcal{T}_Y=\{X\in \mathcal{T}\; |\; F_{X, \Sigma^n(Y)} \mbox{ is an isomorphism, for any } n\in \mathbb{Z}\}.$$
Since $F$ preserves coproducts, it is a localizing subcategory. By Step 1, $\mathcal{U}_+\subseteq \mathcal{T}_Y$. The claim above implies that $\mathcal{T}=\mathcal{T}_Y$. In other words, $F_{X, Y}$ is an isomorphism for any $X\in \mathcal{T}$ and $Y\in \mathcal{U}_{-}$.

Dually, we fix an arbitrary object $X\in \mathcal{T}$ and consider the following full subcategory  of $\mathcal{T}$:
$$_X\mathcal{T}=\{Y\in \mathcal{T}\; |\; F_{X, \Sigma^n(Y)} \mbox{ is an isomorphism, for any } n\in \mathbb{Z}\}.$$
It is a colocalizing subcategory, which contains $\mathcal{U}_{-}$ by the preceding paragraph. The claim above implies that $\mathcal{T}={_X\mathcal{T}}$. This completes the whole proof.
\end{proof}

\section{The pseudo-identities}\label{sec:ps}

In this section, we introduce the notion of  pseudo-identity on a triangulated category with respect to a pre-weight structure. We apply results in Section~\ref{sec:triv-act} to obtain sufficient conditions on when a triangle endofunctor is isomorphic to a pseudo-identity; see Theorem~\ref{thm:pseudo}. In Proposition~\ref{prop:hereditary+ps}, we prove that any pseudo-identity on the derived category of a hereditary abelian category is isomorphic to the genuine identity functor.

Let $\mathcal{T}$ be a triangulated category with a pre-weight structure $(\mathcal{U}_{\geq 0}, \mathcal{U}_{\leq 0})$ whose core is $\mathcal{C}$. The following notion is inspired by \cite[Section~3]{CY}.

\begin{defn}\label{defn:pseudo}
A triangle autoequivalence $F\colon \mathcal{T}\rightarrow \mathcal{T}$ is called a \emph{pseudo-identity} on $\mathcal{T}$ with respect to $(\mathcal{U}_{\geq 0}, \mathcal{U}_{\leq 0})$, if $F(X)=X$ for each object $X$ and the restriction $F|_\mathcal{C}\colon \mathcal{C}\rightarrow \mathcal{C}$ equals ${\rm Id}_\mathcal{C}$ as a functor.
 \end{defn}

 In view of Section~\ref{sec:exm}, the general notion above unifies pseudo-identities both on the homotopy category and derived category in \cite{CY}.  We observe that a pseudo-identity on $\mathcal{T}$ is a triangle automorphism of $\mathcal{T}$, whose inverse is also a pseudo-identity.

The following observation is standard; compare \cite[Corollary~3.4]{CY}.

\begin{lem}\label{lem:pseudo}
Let $F\colon \mathcal{T}\rightarrow \mathcal{T}$ be a triangle endofunctor. Then $F$ is isomorphic to a pseudo-identity if and only if $F$ is a triangle autoequivalence such that $F(X)\simeq X$ for each object $X$ and that the restriction $F|_\mathcal{C}\colon \mathcal{C}\rightarrow \mathcal{C}$ is isomorphic to ${\rm Id}_\mathcal{C}$.
\end{lem}

\begin{proof}
The ``if" part is clear. To prove the ``only if" part, we assume that $(F, \omega)$ is the given triangle endofunctor. Denote by $\eta\colon F|_\mathcal{C}\rightarrow {\rm Id}_\mathcal{C}$ an isomorphism of endofunctors. For each object $X$ in $\mathcal{T}$, we fix an arbitrary isomorphism $\phi_X\colon F(X)\rightarrow X$; moreover, we choose $\phi_X$ to be $\eta_X$ when $X$ belongs to $\mathcal{C}$. Here, we use the  axiom of global choice, that is, the axiom of choice for classes.

We define a triangle endofunctor $(G, \omega')$ on $\mathcal{T}$ as follows. For each object $X$, we set $G(X)=X$; for any morphism $f\colon X\rightarrow Y$, we set $G(f)=\phi_Y\circ F(f)\circ (\phi_X)^{-1}$. Moreover, we define $\omega'\colon G\Sigma\rightarrow \Sigma G$ such that $\omega'_X=\Sigma(\phi_X) \circ \omega_X\circ (\phi_{\Sigma(X)})^{-1}$. It is routine to check that $(G, \omega')$ is a pseudo-identity, which is isomorphic to $(F, \omega)$ as a triangle functor.
\end{proof}

The following criterion is based on Section~\ref{sec:triv-act}.

\begin{thm}\label{thm:pseudo}
Let $\mathcal{T}$ be a bi-complete triangulated category with a pre-weight structure $(\mathcal{U}_{\geq 0}, \mathcal{U}_{\leq 0})$, and let $F\colon \mathcal{T}\rightarrow \mathcal{T}$ be a bi-continuous triangle endofunctor. Assume that the following conditions are fulfilled.
\begin{enumerate}
\item The pre-weight structure $(\mathcal{U}_{\geq 0}, \mathcal{U}_{\leq 0})$ is two-sided convergent and has finite dimension.
\item We have ${\rm Hom}_\mathcal{T}(\mathcal{C}, \Sigma^{-n}(\mathcal{C}))=0$ for any $n\geq 1$.
\item  We have $\mathcal{U}_{\leq 0}\subseteq {\rm Loc}\langle \mathcal{U}_b\rangle$ and $\mathcal{U}_{\geq 0}\subseteq {\rm coLoc}\langle \mathcal{U}_b\rangle$.
\item There is a natural isomorphism $\eta_A\colon F(A)\simeq A$ for each object $A\in \mathcal{C}$.
\item The restriction $F|_{\mathcal{U}_b}\colon \mathcal{U}_b\rightarrow \mathcal{T}$ is fully faithful.
\end{enumerate}
Then $F$ is isomorphic to a pseudo-identity on $\mathcal{T}$.
\end{thm}

\begin{proof}
In view of (1), (2) and (4), we can apply Theorem~\ref{thm:trivial-action} and infer that $F(X)\simeq X$ for each object $X$. In particular, the functor  $F$ is dense. Since $(\mathcal{U}_{\geq 0}, \mathcal{U}_{\leq 0})$ is finite dimensional, it follows that the functor $F$ is bounded with respect to $(\mathcal{U}_{\geq 0}, \mathcal{U}_{\leq 0})$. In view of  (3) and (5),  we  can apply Theorem~\ref{thm:extend-ff} and  infer that $F$ is full faithful. Consequently, $F$ is an autoequivalence. Now the result follows immediately from Lemma~\ref{lem:pseudo}.
\end{proof}

The homotopy category and derived category of unbounded complexes are our main concern. Let $\mathcal{A}$ be an additive category. Recall from Example~\ref{exm:K} that $\mathbf{K}(\mathcal{A})$ has a canonical pre-weight structure $(\mathcal{K}_{\geq 0}, \mathcal{K}_{\leq 0})$. By a pseudo-identity on $\mathbf{K}(\mathcal{A})$, we mean a pseudo-identity on $\mathbf{K}(\mathcal{A})$ with respect to $(\mathcal{K}_{\geq 0}, \mathcal{K}_{\leq 0})$.

\begin{prop}\label{prop:pseudo-K}
Assume that $\mathcal{A}$ has arbitrary coproducts and products, and that $F\colon \mathbf{K}(\mathcal{A})\rightarrow \mathbf{K}(\mathcal{A})$ is a bi-continuous triangle endofunctor. Assume further that $F(\mathcal{A})\subseteq \mathcal{A}$ and that $F|_\mathcal{A}$ is isomorphic to ${\rm Id}_\mathcal{A}$. Then $F$ is isomorphic to a pseudo-identity on $\mathbf{K}(\mathcal{A})$.
\end{prop}

\begin{proof}
We prove the result by verifying the conditions in Theorem~\ref{thm:pseudo}. The canonical pre-weight structure $(\mathcal{K}_{\geq 0}, \mathcal{K}_{\leq 0})$ is of dimension $0$. In view of Lemma~\ref{lem:can-pws-K}, (1) and (3) hold. We identify the core with $\mathcal{A}^\natural$. Then Condition (2) is trivial, and (4) is immediate from the assumptions. For (5), we identify $\mathcal{K}_b$ with $\mathbf{K}^b(\mathcal{A})$, since  the latter category is idempotent-complete; consult the proof of \cite[Proposition~3.2]{BN}. By \cite[Lemma~3.1(1)]{CY}, the restricted functor $F|_{\mathbf{K}^b(\mathcal{A})}\colon \mathbf{K}^b(\mathcal{A})\rightarrow \mathbf{K}(\mathcal{A})$ is fully faithful.
\end{proof}

 Let $\mathcal{A}$ be an abelian category. Recall from Example~\ref{exm:D} the canonical pre-weight structure $(\mathcal{D}_{\geq 0}, \mathcal{D}_{\leq 0})$ on $\mathbf{D}(\mathcal{A})$. By a pseudo-identity on $\mathbf{D}(\mathcal{A})$, we mean a pseudo-identity on $\mathbf{D}(\mathcal{A})$ with respect to $(\mathcal{D}_{\geq 0}, \mathcal{D}_{\leq 0})$.

 \begin{prop}\label{prop:pseudo-D}
Assume that $\mathcal{A}$ is an abelian category with arbitrary coproducts and products such that coproducts and products of exact sequences are still exact. Let $F\colon \mathbf{D}(\mathcal{A})\rightarrow \mathbf{D}(\mathcal{A})$ is a bi-continuous triangle endofunctor. Assume further that $F(\mathcal{A})\subseteq \mathcal{A}$ and that $F|_\mathcal{A}$ is isomorphic to ${\rm Id}_\mathcal{A}$. Then $F$ is isomorphic to a pseudo-identity on $\mathbf{D}(\mathcal{A})$, provided that one of the following conditions holds.
\begin{enumerate}
\item The category $\mathcal{A}$ has enough projective objects or enough injective objects, and ${\rm gl.dim}(\mathcal{A})$ is finite.
\item The given endofunctor $F$ is an autoequivalence.
\end{enumerate}
 \end{prop}

 \begin{proof}
 By Lemma~\ref{lem:can-pws-D}, the pre-weight structure $(\mathcal{D}_{\geq 0}, \mathcal{D}_{\leq 0})$ is two-sided convergent and satisfies Condition (3) in Theorem~\ref{thm:pseudo}. Moreover, it is well known that ${\rm Hom}_{\mathbf{D}(\mathcal{A})}(\mathcal{A}, \Sigma^{-n}(\mathcal{A}))=0$ for $n\geq 1$. Then we apply Theorem~\ref{thm:trivial-action} to infer that $F(X)\simeq X$ for any complex $X$. In presence of (2), we are done by Lemma~\ref{lem:pseudo}.

 If (1) holds, we observe by Lemma~\ref{lem:can-pws-D-fd} that $(\mathcal{D}_{\geq 0}, \mathcal{D}_{\leq 0})$ has finite dimension. In view of \cite[Lemma~3.6]{CY}, we observe that $F|_{\mathbf{D}^b(\mathcal{A})}\colon \mathbf{D}^b(\mathcal{A})\rightarrow \mathbf{D}(\mathcal{A})$ is fully faithful. Consequently, all the conditions of Theorem~\ref{thm:pseudo} are all fulfilled.
 \end{proof}

 \begin{rem}
 In comparison with Proposition~\ref{prop:pseudo-K}, we do not know whether Condition (1) or (2) in Proposition~\ref{prop:pseudo-D} can be removed.
 \end{rem}

 The following result might be viewed as a unbounded version of \cite[Corollary~5.6]{CY}.

\begin{prop}\label{prop:hereditary+ps}
Let $\mathcal{H}$ be a hereditary abelian category. Then any pseudo-identity on $\mathbf{D}(\mathcal{H})$ is isomorphic to the identity functor.
\end{prop}

\begin{proof}
Assume that $(F, \omega)\colon \mathbf{D}(\mathcal{H})\rightarrow \mathbf{D}(\mathcal{H})$ is a pseudo-identity. Recall that for all $n\geq 0$, the natural isomorphisms $\omega^n\colon F\Sigma^n\rightarrow \Sigma^n F$ are defined such that $\omega^0={\rm Id}_F$, $\omega^1=\omega$ and $\omega^{n+1}=(\Sigma\omega^{n})\circ \omega{\Sigma^n}$. Furthermore, we define natural isomorphisms $\omega^{-n}\colon F\Sigma^{-n}\rightarrow \Sigma^{-n}F$ for $n\geq 1$ such that $\omega^{-1}=(\Sigma^{-1}\omega{\Sigma^{-1}})^{-1}$ and $\omega^{-n-1}=(\Sigma^{-n}\omega^{-1})\circ (\omega^{-n}{\Sigma^{-1}})$. For any integers $m$ and $n$, we have
\begin{align}\label{equ:here1}
    \omega^m=(\Sigma^n\omega^{m-n})\circ (\omega^n\Sigma^{m-n}).
\end{align}

Set $\mathcal{C}=\bigcup_{n\in \mathbb{Z}}\Sigma^n(\mathcal{H})$, which is viewed as a full subcategory of $\mathbf{D}(\mathcal{H})$. We will define a natural isomorphism $\eta\colon F|_{\mathcal{C}}\rightarrow {\rm Id}_\mathcal{C}$. For any object $\Sigma^n(A)\in \mathcal{C}$ with $A\in \mathcal{H}$, we define
$$\eta_{\Sigma^n(A)}=\omega^n_A\colon F(\Sigma^n(A))\longrightarrow \Sigma^n F(A)=\Sigma^n(A).$$
For the naturalness of $\eta$, we take a morphism $\Sigma^n(A)\rightarrow \Sigma^m(B)$. Since $\mathcal{H}$ is hereditary, we may assume that $m=n$ or $m=n+1$. The morphism  will be assumed to be $\Sigma^n(\xi)$ for $\xi\colon A\rightarrow \Sigma^{m-n}(B)$. We infer the required naturalness from the following identity.
\begin{align*}
\omega^m_B\circ F\Sigma^n(\xi) &= \Sigma^n (\omega^{m-n}_B)\circ \omega^n_{\Sigma^{m-n}(B)}\circ F\Sigma^n(\xi)\\
&=\Sigma^n(\omega^{m-n}_B) \circ \Sigma^n F(\xi)\circ \omega_A^n\\
&=\Sigma^n(\xi)\circ \omega_A^n
\end{align*}
Here, the first equality uses (\ref{equ:here1}), the second one uses the naturalness of $\omega^n$ and the last one uses $\xi=\omega^{m-n}_B\circ F(\xi)$, which is a special case of \cite[Proposition~5.5]{CY}.

It is well known that each object $X$ in $\mathbf{D}(\mathcal{H})$ is isomorphic to $\coprod_{n\in \mathbb{Z}} \Sigma^{-n}(H^{n}(X))$, where the latter is canonically isomorphic to the product $\prod_{n\in \mathbb{Z}} \Sigma^{-n}(H^{n}(X))$; see \cite[Theorem~A.1]{PS}. Applying Lemma~\ref{lem:extend} below, we extend $\eta$ to a natural isomorphism $\eta'\colon F\rightarrow {\rm Id}_{\mathbf{D}(\mathcal{H})}$. Moreover, since $\eta\Sigma=(\Sigma \eta)\circ \omega$, the uniqueness of the extension implies that $\eta'\Sigma=(\Sigma \eta')\circ \omega$. We infer that $\eta'$ is a natural isomorphism between the two triangle functors $(F, \omega)$ and $({\rm Id}_{\mathbf{D}(\mathcal{H})}, {\rm Id}_\Sigma)$.
\end{proof}

We need a general result on extending natural transformations. Let $F, G\colon \mathcal{A}\rightarrow \mathcal{B}$ be two additive functors between two additive categories, and let $\mathcal{C}$ be a full subcategory of $\mathcal{A}$.  For each object $A$ in $\mathcal{A}$, we assume that there is a \emph{special} isomorphism with respect to $(\mathcal{C}; F, G)$.
 $$\xi_A\colon A\stackrel{\sim}\longrightarrow \coprod_{i\in I}C_i$$
The speciality means the following facts:  these $C_i$'s belong to $\mathcal{C}$ and $I$ is an index set which is not necessarily finite; they are required to satisfy the following conditions.
 \begin{enumerate}
\item[(S1)]  Both the coproduct $\coprod_{i\in I}C_i$ and product $\prod_{i\in I}C_i$ exist in $\mathcal{A}$, and the canonical morphism $\coprod_{i\in I}C_i\rightarrow \prod_{i\in I}C_i$ is an isomorphism.
\item[(S2)] The coproducts $\coprod_{i\in I}F(C_i)$ and $\coprod_{i\in I}G(C_i)$ exist in $\mathcal{B}$, and the canonical morphisms $\coprod_{i\in I}F(C_i)\rightarrow F(\coprod_{i\in I}C_i)$ and  $\coprod_{i\in I}G(C_i)\rightarrow G(\coprod_{i\in I}C_i)$ are isomorphisms.
\item[(S3)] The products $\prod_{i\in I}F(C_i)$ and $\prod_{i\in I}G(C_i)$ exist in $\mathcal{B}$, and the canonical morphisms $F(\prod_{i\in I}C_i)\rightarrow \prod_{i\in I}F(C_i)$ and  $G(\prod_{i\in I}C_i)\rightarrow \prod_{i\in I}G(C_i)$ are isomorphisms.
 \end{enumerate}
The following result slightly generalizes \cite[Lemma~2.4]{CY} with almost the same proof.

\begin{lem}\label{lem:extend}
 Keep the assumptions above. Let $\eta\colon F|_\mathcal{C}\rightarrow G|_\mathcal{C}$ be a natural transformation. Then there is a unique natural transformation $\eta'\colon F\rightarrow G$ extending $\eta$. If $\eta$ is a natural isomorphism, so is $\eta$.
\end{lem}

\begin{proof}
For each object $A$ in $\mathcal{A}$, we fix a special isomorphism $\xi_A\colon A\rightarrow \coprod_{i\in I}C_i$ with respect to $(\mathcal{C}; F, G)$. We make the choice such that $\xi_C$ is the identity morphism for each object $C$ in $\mathcal{C}$. Here, we use the axiom of choice for the class of objects in $\mathcal{A}$.

We define $\eta'_A\colon F(A)\rightarrow G(A)$ to be
$$G(\xi_A^{-1})\circ (\coprod_{i\in I}\eta_{C_i})\circ F(\xi_A).$$
Here, we identify $F(\coprod_{i\in I}C_i)$ with $\coprod_{i\in I}F(C_i)$, $G(\coprod_{i\in I}C_i)$ with $\coprod_{i\in I}G(C_i)$; consult (S2). In particular, we have $\eta'_C=\eta_C$ for $C\in \mathcal{C}$ by our choice.

We claim that  the morphism  $\eta'_A$ is natural in $A$. For this, we take an arbitrary morphism $f\colon A\rightarrow A'$ in $\mathcal{A}$. For the object $A'$, we have the chosen special isomorphism $\xi_{A'}\colon A'\rightarrow \coprod_{j\in J }C'_j$. Then we have the following commutative diagram.
\[
\xymatrix{
A\ar[d]_-{f}\ar[rr]^-{\xi_A} && \coprod_{i\in I} C_i\ar[d]^-{(f_{ji})_{(i, j)\in I\times J} }\\
A'\ar[rr]^-{\xi_{A'}} && \prod_{j\in J} C'_j
}
\]
Here, each $f_{ji}\colon C_i\rightarrow C'_j$ is a morphism in $\mathcal{C}$, and we implicitly identify $\coprod_{j\in J}C'_j$ with $\prod_{j\in J} C'_j$; consult (S1).  The middle square in the following commutative diagram uses the naturalness of $\eta$.
\[
\xymatrix{
F(A) \ar[d]_-{F(f)} \ar[r]^-{F(\xi_A)} & \coprod_{i\in I} F(C_i)\ar[d]^-{(F(f_{ji}))}  \ar[rr]^-{\coprod_{i\in I}\eta_{C_i}} && \coprod_{i\in I}G(C_i)  \ar[d]^-{(G(f_{ji}))} \ar[r]^-{G(\xi_A^{-1})} & G(A)\ar[d]^-{G(f)}\\
F(A')  \ar[r]^-{F(\xi_{A'})} & \prod_{j\in J} F(C'_j) \ar[rr]^-{\prod_{j\in J}\eta_{C'_j}} && \prod_{j\in J}G(C'_j) \ar[r]^-{G(\xi_{A'}^{-1})} & G(A')
}\]
The outer commutative  diagram proves that $\eta'_{A'}\circ F(f)=G(f)\circ \eta'_A$, as required.
\end{proof}

\section{Differential graded endomorphism algebras}\label{sec:dg}

From this section on, we  fix a commutative ring $\mathbb{K}$ and will work over $\mathbb{K}$. In other words, we will require all algebras, categories and functors are $\mathbb{K}$-linear.

In this section, we prove that homotopically equivalent complexes have quasi-isomorphic dg endomorphism algebras; see Proposition~\ref{prop:dg}. This result is partly  known to experts, and will be used in Section~\ref{sec:can}.

For a dg algebra $\Lambda$, we denote by $\mathbf{D}(\Lambda)$ the derived category \cite{Kel94} of left dg $\Lambda$-modules. We always view an ordinary algebra as a dg algebra concentrated in degree zero. Let $B$ be an ordinary algebra. Since  left dg $B$-modules coincide with complexes of $B$-modules, we have $\mathbf{D}(B)=\mathbf{D}(B\mbox{-Mod})$. Denote by $\Lambda^{\rm op}$ the opposite dg algebra of $\Lambda$. We identify left dg $\Lambda$-modules with right dg $\Lambda^{\rm op}$-modules.

For two complexes $P=(P^n, d_P^n)_{n\in \mathbb{Z}}$ and $Q=(Q^n, d_Q^n)_{n\in \mathbb{Z}}$ of $B$-modules, the \emph{Hom-complex} ${\rm Hom}_B(P, Q)$ is defined such that its $p$-th component is given by a countable product
$$ {\rm Hom}_B(P, Q)^p=\prod_{n\in \mathbb{Z}}{\rm Hom}_B(P^n, Q^{n+p}),$$
whose elements are denoted by $f=(f^n)_{n\in \mathbb{Z}}$ with $f^n\colon P^n\rightarrow Q^{n+p}$. Therefore, the elements in ${\rm Hom}_B(P, Q)^p$ are viewed as graded homomorphisms of degree $p$. The differential $\partial^p \colon {\rm Hom}_B(P, Q)^p\rightarrow {\rm Hom}_B(P, Q)^{p+1}$ sends $f$ to $\partial(f)$ such that
$$\partial^p(f)^n=d_Q^{p+n}\circ f^n-(-1)^p f^{n+1}\circ d_P^n.$$
When $P=Q$, ${\rm Hom}_B(P, P)={\rm End}_B(P)$ becomes a dg algebra, whose multiplication is induced by composition of graded homomorphisms.  The resulting dg algebra ${\rm End}_B(P)$ is called the \emph{dg endomorphism algebra} of $P$.

We observe that $P$ is naturally a left dg ${\rm End}_B(P)$-module. More precisely, for $f\in {\rm End}_B(P)$ and $x\in P$, the left action $fx$ is given by $f(x)$. Furthermore, we observe that $P$ is naturally a dg $B$-${\rm End}_B(P)^{\rm op}$-bimodule. Similarly, $Q$ is naturally a dg $B$-${\rm End}_B(Q)^{\rm op}$-bimodule.

The folloiwng consideration is taken from \cite[Subsection~2.9]{Dri}. Let $\xi\colon P\rightarrow Q$ be a cochain morphism between complexes of $B$-modules. We consider the following matrix dg algebra associated to $\xi$.
 $$\Lambda=\begin{pmatrix}
 {\rm End}_B(Q) & \Sigma^{-1}{\rm Hom}_B(P, Q)\\
 0 & {\rm End}_B(P)
 \end{pmatrix}$$
 A typical element of degree $p$ in $\Lambda$ is given by
 $\begin{pmatrix}
 g & s^{-1}h\\
 0 & f
 \end{pmatrix}$
with $g\in  {\rm End}_B(Q)^p$, $h\in {\rm Hom}_B(P, Q)^{p-1}$ and $f\in {\rm End}_B(P)^p$.  Here, $s^{-1}$ is viewed as a symbol of degree one, and we have $s^{-1}h\in (\Sigma^{-1} {\rm Hom}_B(P, Q))^p$. The multiplication is as follows.
 $$\begin{pmatrix}
 g & s^{-1}h\\
 0 & f
 \end{pmatrix} \begin{pmatrix}
 g' & s^{-1}h'\\
 0 & f'
 \end{pmatrix}=\begin{pmatrix}
 g\circ g' & s^{-1}((-1)^{|g|}g\circ h'+h\circ f')\\
 0 & f\circ f'
 \end{pmatrix}.$$
 Here, $|g|$ denotes the degree of $g$. The differential $d_\Lambda$ is given such that
 $$d_\Lambda \begin{pmatrix}
 g & s^{-1}h\\
 0 & f
 \end{pmatrix}=\begin{pmatrix}
 \partial(g) & s^{-1}(-\partial(h)+\xi\circ f-g\circ \xi)\\
 0 & \partial(f)
 \end{pmatrix}.$$
 We have two surjective homomorphisms ${\rm pr}_1\colon \Lambda \rightarrow {\rm End}_B(Q)$ and ${\rm pr}_2\colon \Lambda \rightarrow {\rm End}_B(P)$ of dg algebras such that
 $${\rm pr}_1 \begin{pmatrix}
 g & s^{-1}h\\
 0 & f
 \end{pmatrix}=g \mbox{ and } {\rm pr}_2 \begin{pmatrix}
 g & s^{-1}h\\
 0 & f
 \end{pmatrix}=f.$$
 Using these two homomorpshims, both $P$ and $Q$ become dg $B$-$\Lambda^{\rm op}$-bimodules. Therefore, we view them as objects in $\mathbf{D}(B\otimes \Lambda)$. Here, we identify left dg $B\otimes \Lambda$-modules with dg $B$-$\Lambda^{\rm op}$-bimodules.

 The first statement of the following result is known to experts.

\begin{prop}\label{prop:dg}
Assume that  $\xi\colon P\rightarrow Q$ is a homotopy equivalence. Then the following statements hold.
\begin{enumerate}
\item Both surjections ${\rm pr}_i$ are quasi-isomorphisms.
\item There is an isomorphism $P\simeq Q$ in $\mathbf{D}(B\otimes \Lambda)$.
\end{enumerate}
\end{prop}

\begin{proof}
For (1), we observe that the homotopy equivalence  $\xi$ induces homotopy equivalences
$${\rm End}_B(Q)\longrightarrow {\rm Hom}_B(P, Q) \longleftarrow {\rm End}_B(P)$$
among complexes of $\mathbb{K}$-modules. Then we infer that the kernels of both ${\rm pr}_i$'s are acyclic, and prove (1).

For (2), we consider the mapping cone $C=Q\oplus \Sigma(P)$ of $\xi$, whose typical element of degree $p$ is written as a column vector $\begin{pmatrix} x\\ sy\end{pmatrix}$ with $x\in Q^p$ and $y\in P^{p+1}$. Here, $s$ is viewed as a symbol of degree $-1$. The differential is given such that $d_C\begin{pmatrix} x\\ sy\end{pmatrix}=\begin{pmatrix} d_Q(x)+\xi(y)\\ -sd_P(y)\end{pmatrix}$. We observe that $C$ is a left dg $\Lambda$-module with the action given by
$$ \begin{pmatrix}
 g & s^{-1}h\\
 0 & f
 \end{pmatrix} \begin{pmatrix} x\\ sy\end{pmatrix}= \begin{pmatrix} g(x)+(-1)^{|h|}h(y)\\ (-1)^{|f|}sf(y)\end{pmatrix}.$$
 Since the $\Lambda$-action is compatible with the $B$-module structure, $C$ becomes a dg $B$-$\Lambda^{\rm op}$-bimodule.

 We have the following short exact sequence of dg $B$-$\Lambda^{\rm op}$-bimodules.
\begin{align}\label{ses:dg}
0\longrightarrow Q\xrightarrow{\begin{pmatrix} 1 \\ 0 \end{pmatrix}} C \xrightarrow{\begin{pmatrix}0 & 1\end{pmatrix}} \Sigma(P)\longrightarrow 0
\end{align}
 This  sequence induces an exact triangle
 $$Q\xrightarrow{\begin{pmatrix} 1 \\ 0 \end{pmatrix}} C \xrightarrow{\begin{pmatrix}0 & 1\end{pmatrix}} \Sigma(P) \dashrightarrow \Sigma(Q)$$
 in $\mathbf{D}(B\otimes \Lambda)$. Since $\xi$ is a homotopy equivalence, $C$ is acyclic and isomorphic to zero in $\mathbf{D}(B\otimes \Lambda)$. Then the dash arrow above is an isomorphism, which yields the desired isomorphism.
\end{proof}

\begin{rem}
Although both $P$ and $Q$ are left dg $\Lambda$-modules, the chain map $\xi$ is not a morphism of dg $\Lambda$-modules in general.  As shown in the proof above, the isomorphism $P\simeq Q$ is obtained by applying $\Sigma^{-1}$ to the dash arrow above, which is given by the following roof.
$$\Sigma(P) \longleftarrow C\oplus \Sigma(Q) \xrightarrow{(0, 1)} \Sigma(Q)$$
Here, the unnamed arrow is induced by the projection $C\rightarrow \Sigma(P)$. We emphasize that the exact sequence (\ref{ses:dg}) does not split in the category of graded $\Lambda$-modules in general.
\end{rem}

\section{Canonical derived equivalences}\label{sec:can}

In this section, we introduce the canonical derived equivalence associated to a tilting complex and an algebra isomorphism. We give a criterion on when two canonical derived equivalences are equivalent; see Theorem~\ref{thm:can}. A general factorization theorem is proved in Theorem~\ref{thm:factor}. We prove in Theorem~\ref{thm:hereditary+can} that any derived equivalence starting from a left hereditary algebra is always canonical.

Let $A$ be a $\mathbb{K}$-algebra, which is not necessarily flat. Denote by $A\mbox{-Mod}$ the abelian category of left $A$-modules, and by $\mathbf{D}(A\mbox{-Mod})$ its derived category. For each element $a\in A$, we denote by $r_a\colon A\rightarrow A$ the morphism in $A\mbox{-Mod}$ sending $x$ to $xa$. Since we identify a module with the corresponding stalk complex concentrated in degree zero, each $r_a$ is viewed as a morphism in $\mathbf{D}(A\mbox{-Mod})$.

The following result is analogous to \cite[Proposition~5.1]{CC}.

\begin{lem}\label{lem:pseudo-D}
Let $F\colon \mathbf{D}(A\mbox{-}{\rm Mod})\rightarrow \mathbf{D}(A\mbox{-}{\rm Mod})$ be a triangle endofunctor. Then $F$ is isomorphic to a pseudo-identity on $\mathbf{D}(A\mbox{-}{\rm Mod})$  if and only if $F$ is an auto-equivalence such that there is an isomorphism $\theta\colon F(A)\rightarrow A$ in $\mathbf{D}(A\mbox{-}{\rm Mod})$ satisfying $r_a\circ \theta=\theta\circ F(r_a)$ for each $a\in A$.
\end{lem}

\begin{proof}
The ``only if" part is direct. For the ``if" part, we take any $A$-module $M$. For each $n\neq 0$, we have  the following isomorphisms of $\mathbb{K}$-modules.
\begin{align*}
H^n F(M) &\simeq {\rm Hom}_{\mathbf{D}(A\mbox{-}{\rm Mod})}(A, \Sigma^n F(M))\\
&\simeq {\rm Hom}_{\mathbf{D}(A\mbox{-}{\rm Mod})}(F(A), \Sigma^n F(M))\\
&\simeq {\rm Hom}_{\mathbf{D}(A\mbox{-}{\rm Mod})}(A, \Sigma^n(M))=0
\end{align*}
Here, the second isomorphism uses the isomorphism $\theta$ and the last one uses the fully-faithfulness of $F$. This implies that $F(M)$ is isomorphic to a stalk complex concentrated in degree zero. Therefore, we have
$F(A\mbox{-}{\rm Mod})\subseteq A\mbox{-}{\rm Mod}$.

We observe that $F|_{A\mbox{-}{\rm Mod}}\colon A\mbox{-}{\rm Mod}\rightarrow A\mbox{-}{\rm Mod}$ is exact and preserves arbitrary coproducts. It follows from the classical Eilenberg-Watt's theorem that $F|_{A\mbox{-}{\rm Mod}}$ is isomorphic to $F(A)\otimes_A -$. The assumption on $\theta$ implies that the $A$-$A$-bimodule $F(A)$ is isomorphic to $A$. Consequently, $F|_{A\mbox{-}{\rm Mod}}$ is isomorphic to the identity functor. Then the required statement follows directly from Proposition~\ref{prop:pseudo-D}.
\end{proof}

Denote by $B\mbox{-proj}$ the category of finitely generated projective $B$-modules and by $\mathbf{K}^b(B\mbox{-proj})$ its bounded homotopy category. The canonical functor $\mathbf{K}^b(B\mbox{-proj})\rightarrow \mathbf{D}(B\mbox{-Mod})$ is fully faithful, whose essential image is $\mathbf{D}(B\mbox{-Mod})^c$, the subcategory of $\mathbf{D}(B\mbox{-Mod})$ formed by compact objects.

  Recall from \cite{Ric89} that a \emph{tilting complex} $P$ over $B$ is an object $P\in \mathbf{K}^b(B\mbox{-proj})$ such that ${\rm Hom}_{\mathbf{K}^b(B\mbox{-}{\rm proj})}(P, \Sigma^n(P))=0$ for any $n\neq 0$, and that ${\rm Loc}\langle P\rangle=\mathbf{D}(B\mbox{-Mod})$. The latter condition might be replaced by the following one: the smallest triangulated subcategory of $\mathbf{K}^b(B\mbox{-}{\rm proj})$ containing $P$ and closed under direct summands is $\mathbf{K}^b(B\mbox{-}{\rm proj})$ itself.

Recall the dg endomorphism algebra ${\rm End}_B(P)$ of a complex $P$ from Section~\ref{sec:dg}. We will need the truncated dg subalgebra $\tau_{\leq 0}\; {\rm End}_B(P)$ of ${\rm End}_B(P)$, whose underlying graded space is given by $(\bigoplus_{n\geq 1}{\rm End}_B(P)^{-n}) \oplus {\rm Ker}\; \partial^0$. Moreover, we have the canonical map
\begin{align}\label{equ:pi-P}
\pi_P\colon \tau_{\leq 0}\; {\rm End}_B(P)\longrightarrow H^0({\rm End}_B(P))={\rm End}_{\mathbf{K}^b(B\mbox{-}{\rm Mod})}(P).
\end{align}
Denote by $\Gamma=(\tau_{\leq 0}\; {\rm End}_B(P))^{\rm op}$, the opposite dg algebra of $\tau_{\leq 0}\; {\rm End}_B(P)$. We observe that $P$ is naturally a dg $B$-$\Gamma$-bimodule.  We have the corresponding derived tensor functor
\begin{align}\label{equ:derived-tensor}
P\otimes_\Gamma^\mathbb{L}-\colon \mathbf{D}(\Gamma)\longrightarrow \mathbf{D}(B\mbox{-Mod}).
\end{align}

Let $P$ be  a tilting complex over $B$. Then the functor (\ref{equ:derived-tensor}) is a triangle equivalence; see \cite[Section~8]{Kel94}. We observe that $H^n({\rm End}_B(P))=0$ for $n\neq 0$. Therefore, $\pi_P$ is a quasi-isomorphism between dg algebras.

 By a \emph{tilting pair} $(P, \phi)$ from $A$ to $B$, we mean a tilting complex $P$ over $B$ and an algebra isomorphism $\phi\colon A\rightarrow {\rm End}_{\mathbf{K}^b(B\mbox{-}{\rm proj})}(P)^{\rm op}$.

The construction in the following definition is implicit in \cite{Kel93}.

 \begin{defn}
Let $(P, \phi)$ be a tilting pair from $A$ to $B$. The \emph{canonical derived equivalence} associated to $(P, \phi)$ is the following composite triangle equivalence.
$$\Psi_{(P, \phi)}\colon \mathbf{D}(A\mbox{-Mod}) \xrightarrow{(\phi^{-1}\circ \pi_P)^*} \mathbf{D}(\Gamma)\xrightarrow{P\otimes^\mathbb{L}_\Gamma-} \mathbf{D}(B\mbox{-Mod}).$$
Here, $\Gamma=(\tau_{\leq 0}\; {\rm End}_B(P))^{\rm op}$ and the functor $(\phi^{-1}\circ \pi_P)^*$ is induced by the dg algebra quasi-isomorphism $\phi^{-1}\circ \pi_P\colon \Gamma\rightarrow A$.

In general, a derived equivalence $F\colon \mathbf{D}(A\mbox{-Mod})\rightarrow \mathbf{D}(B\mbox{-Mod})$ is called \emph{canonical} if it is isomorphic to a canonical derived equivalence as a triangle functor.
 \end{defn}

 \begin{rem}\label{rem:canonical}
We observe that $\Psi_{(P, \phi)}(A)=P\otimes_\Gamma^\mathbb{L}A\simeq P\otimes_\Gamma \Gamma=P$, where $A$ is viewed as a left dg $\Gamma$-module via the quasi-isomorphism $\phi^{-1}\circ \pi_P$. In other words, we identify $\Psi_{(P, \phi)}(A)$ with $P$. Under this identification, we observe that $\Psi_{(P, \phi)}(r_a)=\phi(a)$ for any element $a\in A$.
 \end{rem}

 The following technical proof relies on the results in Section~\ref{sec:dg}.

 \begin{thm}\label{thm:can}
 Let $(P, \phi)$ and $(Q, \psi)$ be two tilting pairs from $A$ to $B$. Then $\Psi_{(P, \phi)}\simeq \Psi_{(Q, \psi)}$ if and only if there is an isomorphism $\xi\colon P\rightarrow Q$ in $\mathbf{K}^b(B\mbox{-}{\rm proj})$ satisfying $\psi(a)=\xi\circ \phi(a)\circ \xi^{-1}$ for any $a\in A$.
 \end{thm}

 \begin{proof}
 In view of Remark~\ref{rem:canonical}, $\Psi_{(P, \phi)}$ recovers the tilting pair $(P, \phi)$. Then ``only if" part follows immediately.

 For the ``if" part, we set $\Delta=(\tau_{\leq 0} \; {\rm End}_B(Q))^{\rm op}$. Associated to the tilting complex $Q$, we have the quasi-isomorphism $\pi_Q\colon \Delta \rightarrow {\rm End}_{\mathbf{K}^b(B\mbox{-}{\rm proj})}(Q)^{\rm op}$. We consider the matrix dg algebra $\Lambda$ associated to $\xi$ in Section~\ref{sec:dg}. Since $\xi$ is a homotopy equivalence, the two surjective homomorphisms ${\rm pr}_1\colon \Lambda \rightarrow {\rm End}_B(Q)$ and ${\rm pr}_2\colon \Lambda \rightarrow {\rm End}_B(P)$ are quasi-isomorphisms, which induce quasi-isomorphisms between the truncated dg subalgebras.

 We observe that the following diagram commutes.
\begin{align}\label{com-diag:Gamma}
\xymatrix{
 \Gamma \ar[d]_-{\pi_P} & (\tau_{\leq 0}\; \Lambda)^{\rm op}\ar[l]_-{{\rm pr}_2} \ar[r]^-{{\rm pr}_1} & \Delta\ar[d]^-{\pi_Q}\\
 {\rm End}_{\mathbf{K}^b(B\mbox{-}{\rm proj})}(P)^{\rm op} \ar[rr] && {\rm End}_{\mathbf{K}^b(B\mbox{-}{\rm proj})}(Q)^{\rm op}
 }\end{align}
 Here, the unnamed isomorphism sends $u$ to $\xi\circ u\circ \xi^{-1}$. The commutativity follows immediately from the following fact: any element in $(\tau_{\leq 0}\; \Lambda)^{\rm op}$ of degree zero is of the form  $\begin{pmatrix}
 g & s^{-1}h\\
 0 & f
 \end{pmatrix}$
with $g\in  {\rm End}_B(Q)^0$, $h\in {\rm Hom}_B(P, Q)^{-1}$ and $f\in {\rm End}_B(P)^0$;  since $d_\Lambda$ vanishes on this element, we infer that $g$ and $f$ are cochain maps such that $\xi\circ f= g\circ \xi$ in $\mathbf{K}^b(B\mbox{-}{\rm proj})$.

The commutative diagram (\ref{com-diag:Gamma}) and the assumption imply that
$$\phi^{-1}\circ \pi_P\circ {\rm pr}_2=\psi^{-1}\circ \pi_Q\circ {\rm pr}_1.$$
Denote the common value by $\eta\colon (\tau_{\leq 0}\; \Lambda)^{\rm op}\rightarrow A$.

We claim that $\Psi_{(P, \phi)}$ is isomorphic to the following composite functor
$$\mathbf{D}(A\mbox{-Mod}) \stackrel{\eta^*}\longrightarrow \mathbf{D}((\tau_{\leq 0}\; \Lambda)^{\rm op})\xrightarrow{P\otimes^\mathbb{L}_{(\tau_{\leq 0}\; \Lambda)^{\rm op}}-} \mathbf{D}(B\mbox{-Mod}).$$
Recall that $\eta^*=({\rm pr}_2)^*\circ (\phi^{-1}\circ \pi_P)^*$ and that $\Gamma\otimes^\mathbb{L}_{(\tau_{\leq 0}\; \Lambda)^{\rm op}}-$ is a quasi-inverse of ${\rm pr}_2)^*$. Then the claim follows immediately from the following canonical isomorphism
$$(P\otimes^\mathbb{L}_\Gamma-) \circ (\Gamma\otimes^\mathbb{L}_{(\tau_{\leq 0}\; \Lambda)^{\rm op}}-) \simeq P\otimes^\mathbb{L}_{(\tau_{\leq 0}\; \Lambda)^{\rm op}}-.$$

Similarly, we have that $\Psi_{(Q, \psi)}$ is isomorphic to the following composite functor
$$\mathbf{D}(A\mbox{-Mod}) \stackrel{\eta^*}\longrightarrow \mathbf{D}((\tau_{\leq 0}\; \Lambda)^{\rm op})\xrightarrow{Q\otimes^\mathbb{L}_{(\tau_{\leq 0}\; \Lambda)^{\rm op}}-} \mathbf{D}(B\mbox{-Mod}).$$
By Proposition~\ref{prop:dg}(2), the two dg $B$-$(\tau_{\leq 0}\; \Lambda)^{\rm op}$-bimodules  $P$ and $Q$ are isomorphic in $\mathbf{D}(B\otimes (\tau_{\leq 0}\; \Lambda))$. Therefore, the derived tensor functors
$P\otimes^\mathbb{L}_{(\tau_{\leq 0}\; \Lambda)^{\rm op}}-$ and $Q\otimes^\mathbb{L}_{(\tau_{\leq 0}\; \Lambda)^{\rm op}}-$ are isomorphic. From this, we conclude that $\Psi_{(P, \phi)}$ and $\Psi_{(Q, \psi)}$ are isomorphic.
 \end{proof}

 The following general factorization theorem might be traced back to \cite[Proposition~7.1]{Ric89}  and \cite[Corollary~3.5]{Ric91}; compare \cite[Proposition~5.8]{CY} and \cite[Theorem~5.4]{CC}.

 \begin{thm}\label{thm:factor}
 Let $F\colon \mathbf{D}(A\mbox{-}{\rm Mod})\rightarrow \mathbf{D}(B\mbox{-}{\rm Mod})$ be a derived equivalence. Then there is  an isomorphism $F\simeq F_2 F_1$ with $F_1$ a pseudo-identity on $\mathbf{D}(A\mbox{-}{\rm Mod})$ and $F_2$ a canonical derived equivalence. Such $F_1$ and $F_2$ are unique up to isomorphism.
 \end{thm}

 \begin{proof}
 It is well known that $F(A)$ is isomorphic to a tilting complex in $\mathbf{D}(B\mbox{-Mod})$. We take a tilting complex $P$ in $\mathbf{K}^b(B\mbox{-proj})$ and fix an isomorphism $\zeta\colon P\rightarrow F(A)$. Then we have an algebra isomorphism $\phi\colon A\rightarrow {\rm End}_{\mathbf{K}^b(B\mbox{-}{\rm proj})}(P)^{\rm op}$ given by $\phi(a)=\zeta^{-1}\circ F(r_a)\circ \zeta$ for any $a\in A$.

 Consider the canonical derived equivalence $\Psi_{(P, \phi)}$ associated to the tilting pair $(P, \phi)$. Set $F_1=F^{-1}\Psi_{(P, \phi)}$ with $F^{-1}$ a quasi-inverse of $F$. We have an isomorphism
 $$\theta\colon F_1(A)=F^{-1}(P)\xrightarrow{F^{-1}(\zeta)}  F^{-1}F(A)\simeq A$$
 in $\mathbf{D}(A\mbox{-Mod})$.  Using  Remark~\ref{rem:canonical}, it is routine to verify that $\theta$ satisfies the condition in Lemma~\ref{lem:pseudo-D}. It follows that $F_1$ is isomorphic to a pseudo-identity, and so is $F_1^{-1}$. This proves the existence of such a factorization $F\simeq \Psi_{(P, \phi)}F_1^{-1}$.

 For the uniqueness, we take two tilting pairs $(P, \phi)$ and $(Q, \psi)$ from $A$ to $B$. We assume that there is a natural  isomorphism $\alpha\colon \Psi_{(P, \phi)}F_1\simeq \Psi_{(Q, \psi)}F'_1$, where $F_1$ and $F'_1$ are pseudo-identities. Applying $\alpha$ on $A$, we obtain an isomorphism $\xi\colon P\rightarrow Q$. Since $\alpha$ is natural, we infer that $\xi$ satisfies $\psi(a)=\xi\circ \phi(a)\circ \xi^{-1}$ for any $a\in A$. Theorem~\ref{thm:can} implies that $\Psi_{(P, \phi)}\simeq \Psi_{(Q, \psi)}$, which further implies that $F_1\simeq F'_1$.
 \end{proof}

 The following consequence might be proved directly by a similar argument in the preceding paragraph.

 \begin{cor}\label{cor:can+ps}
 Let $F\colon \mathbf{D}(A\mbox{-}{\rm Mod})\rightarrow \mathbf{D}(A\mbox{-}{\rm Mod})$ be a derived auto-equivalence. If $F$ is canonical and is isomorphic to a pseudo-identity, then $F$ is isomorphic to the identity functor.
 \end{cor}

 \begin{proof}
 We have two factorizations $F={\rm Id}_{\mathbf{D}(A\mbox{-}{\rm Mod})} F=F{\rm Id}_{\mathbf{D}(A\mbox{-}{\rm Mod})}$ of $F$. The uniqueness in Theorem~\ref{thm:factor} implies the required isomorphism.
 \end{proof}

  \begin{cor}\label{cor:can=ps}
 Let $A$ and $B$ be two algebras. Then the following two statements are equivalent.
 \begin{enumerate}
 \item Any derived equivalence $\mathbf{D}(A\mbox{-}{\rm Mod})\rightarrow \mathbf{D}(B\mbox{-}{\rm Mod})$ is canonical.
 \item Any pseudo-identity on $\mathbf{D}(A\mbox{-}{\rm Mod})$ is isomorphic to the identity functor.
 \end{enumerate}
 \end{cor}

 \begin{proof}
 The implication ``(1) $\Rightarrow$ (2)" follows from Corollary~\ref{cor:can+ps}, and the converse follows from the existence of the factorization in  Theorem~\ref{thm:factor}.
 \end{proof}

  In view of  \cite[the remarks after~Definition~3.4]{Ric91}, it is natural to propose the following conjecture.

\begin{conj}\label{conj:can}
Any derived equivalence between algebras is canonical.
\end{conj}

Corollary~\ref{cor:can=ps} implies that  Conjecture~\ref{conj:can} is equivalent to the following question: is any pseudo-identity on the derived category of a module category is isomorphic to the genuine identity functor; compare \cite[Conjecture~5.11]{CY}.

Recall that an algebra $A$ is \emph{left hereditary} if its left global dimension is at most one. This is equivalent to the condition that the module category $A\mbox{-Mod}$ is a hereditary abelian category.

\begin{thm}\label{thm:hereditary+can}
Let $A$ and $B$ be two algebras with $A$ being left hereditary. Then any derived equivalence $F\colon \mathbf{D}(A\mbox{-}{\rm Mod})\rightarrow \mathbf{D}(B\mbox{-}{\rm Mod})$ is canonical.
\end{thm}

\begin{proof}
We just combine Proposition~\ref{prop:hereditary+ps} and Corollary~\ref{cor:can=ps}.
\end{proof}

\begin{rem}
We apply Theorem~\ref{thm:hereditary+can} to the setting considered in \cite{KZ}. It follows that any derived equivalence starting from a hereditary order is canonical. This seems to work in concert with \cite[the sixth paragraph in p.1898]{KZ}, which claims that such a derived equivalence usually is not induced by a two-sided tilting complex. This observation somehow motivates the notion of canonical derived equivalence.
\end{rem}

In view of Lemma~\ref{lem:stan-comp} below, the following questions seem to be natural.

\begin{ques}\label{ques:1}
Is the composition of two canonical derived equivalences canonical? Is any quasi-inverse of a canonical derived equivalence still canonical?
\end{ques}

Concerning the questions above, we have the following observation.

\begin{prop}
   Assume that the composition of any two canonical derived equivalences is always canonical. Then any quasi-inverse of a canonical derived equivalence is canonical.
\end{prop}

\begin{proof}
   Let $F\colon \mathbf{D}(A\mbox{-}{\rm Mod})\rightarrow \mathbf{D}(B\mbox{-}{\rm Mod})$ be an arbitrary canonical derived equivalence with $F^{-1}$ its quasi-inverse. Theorem~\ref{thm:factor} implies that there is a factorization $F^{-1}\simeq F_2F_1$ with $F_1$ a pseudo-identity on $\mathbf{D}(B\mbox{-}{\rm Mod})$ and $F_2$ canonical. We have a natural isomorphism $FF_2\simeq F_1^{-1}$. We observe that $F_1^{-1}$ is a pseudo-identity and that by the assumption $FF_2$ is canonical. By Corollary~\ref{cor:can+ps}, we infer $F_1^{-1}$ is isomorphic to the identity functor, and thus so is $F_1$. Therefore, we have a natural isomorphism $F^{-1}\simeq F_2$, which implies that $F^{-1}$ is canonical.
\end{proof}

\section{Standard derived equivalences}\label{sec:stan}

In this section, we prove that for flat algebras,  canonical derived equivalences coincide with  standard derived equivalences in the sense of \cite{Ric91}; see Theorem~\ref{thm:stan=can}.

Let $A$ be a $\mathbb{K}$-algebra. Denote by $A^e=A\otimes A^{\rm op}$ its \emph{enveloping algebra}. Here, the unadorned tensor means $\otimes_\mathbb{K}$. We identify $A$-$A$-bimodules with $A^e$-modules.

The following fact is well known.

\begin{lem}\label{lem:bimod-A}
Assume that $Z$ is a complex of $A$-$A$-bimodules such that $Z\otimes^\mathbb{L}_A-\colon \mathbf{D}(A\mbox{-}{\rm Mod})\rightarrow \mathbf{D}(A\mbox{-}{\rm Mod})$ is an auto-equivalence. Then we have an isomorphism $Z\simeq A$ in $\mathbf{D}(A^e\mbox{-}{\rm Mod})$ if and only if  ${Z\otimes^\mathbb{L}_A-}|_{A\mbox{-}{\rm proj}} \colon A\mbox{-}{\rm proj}\rightarrow \mathbf{D}(A\mbox{-}{\rm Mod})$ is isomorphic to the inclusion $A\mbox{-}{\rm proj}\hookrightarrow \mathbf{D}(A\mbox{-}{\rm Mod})$.
\end{lem}

\begin{proof}
The ``only if" part is clear, since any isomorphism  $Z\simeq A$ in $\mathbf{D}(A^e\mbox{-}{\rm Mod})$ induces an isomorphism $Z\otimes^\mathbb{L}_A-\simeq A\otimes^\mathbb{L}_A-$, which is isomorphic to the identity functor.

For the ``if" part, we write $F=Z\otimes^\mathbb{L}_A-$. The assumption implies that $F(A)=Z$ is isomorphic to $A$ in $\mathbf{D}(A\mbox{-}{\rm Mod})$. It follows that $H^n(Z)=0$ for $n\neq 0$. Therefore, we may assume that $Z$ is an $A$-$A$-bimodule, which is viewed as a complex concentrated in degree zero. We mention that $Z$ is isomorphic to $A$ as a left $A$-module.

 For any $A$-module $M$ and any $n\neq 0$, we have isomorphisms
\begin{align*}
H^n(FM) &\simeq {\rm Hom}_{\mathbf{D}(A\mbox{-}{\rm Mod})}(A, \Sigma^n(FM))\\
&\simeq {\rm Hom}_{\mathbf{D}(A\mbox{-}{\rm Mod})}(F(A), \Sigma^n(FM))\\
&\simeq {\rm Hom}_{\mathbf{D}(A\mbox{-}{\rm Mod})}(A, \Sigma^n(M))=0.
\end{align*}
Here, the last isomorphism uses the fact that $F$ is an auto-equivalence.  It follows from the isomorphisms above that $F(M)=Z\otimes^\mathbb{L}_A M$ is always concentrated in degree zero. In particular, we have ${\rm Tor}_n^A(Z, M)=0$ for $n\geq 1$. We infer that $Z$ is a flat right $A$-module.  Hence, $F$ is isomorphic to the usual  tensor functor $Z\otimes_A-$ applied on complexes. Therefore, the assumption implies that $Z\otimes_A-\colon A\mbox{-}{\rm proj}\rightarrow A\mbox{-}{\rm proj}$ is isomorphic to the identity functor. From this, we infer infer that $Z$ is isomorphic to $A$ as an $A$-$A$-bimodule.
\end{proof}

In what follows, we assume that both $A$ and $B$ are flat $\mathbb{K}$-algebras.

For a complex $X$ of $B$-$A$-bimodules and a complex $Y$ of $A$-$B$-bimodules, the derived tensor product $X\otimes^\mathbb{L}_A Y$ is naturally a complex of $B$-$B$-bimodules. Moreover, for any complex $N$ of $B$-modules, we have a natural isomorphism
\begin{align}\label{iso:XY}
X\otimes^\mathbb{L}_A (Y\otimes^\mathbb{L}_B N)\simeq (X\otimes^\mathbb{L}_A Y)\otimes^\mathbb{L}_B N
\end{align}
 in $\mathbf{D}(B\mbox{-Mod})$. Similarly, for a complex $M$ of $A$-modules, we have a natural isomorphism
\begin{align}\label{iso:YX}
Y\otimes^\mathbb{L}_B (X\otimes^\mathbb{L}_A M)\simeq (Y\otimes^\mathbb{L}_B X)\otimes^\mathbb{L}_A M
\end{align}
in  $\mathbf{D}(A\mbox{-Mod})$.

The following result is known at least when the algebras are projective; see \cite[Section~4]{Ric91} and \cite[Theorem~1.6 and Remark~1.12]{Yek}. It seems that in the flat case, the implication ``(1) $\Rightarrow$ (3)" is somehow nontrivial. For resolutions of complexes, we refer to \cite[Section~3]{Kel94} and \cite[Section~5]{Kra}.

\begin{prop}
Let $X$ be a complex of $B$-$A$-bimodules. Then the following statements are equivalent.
\begin{enumerate}
\item The functor $X\otimes^\mathbb{L}_A- \colon \mathbf{D}(A\mbox{-}{\rm Mod})\rightarrow \mathbf{D}(B\mbox{-}{\rm Mod})$ is an equivalence.
\item The functor $\mathbb{R}{\rm Hom}_B(X, -)\colon \mathbf{D}(B\mbox{-}{\rm Mod})\rightarrow \mathbf{D}(A\mbox{-}{\rm Mod})$ is an equivalence.
\item There exists a complex $Y$ of $A$-$B$-bimodules such that there are an isomorphism
$X\otimes_A^\mathbb{L} Y\simeq B$ in $\mathbf{D}(B^e\mbox{-}{\rm Mod})$ and an isomorphism $Y\otimes_B^\mathbb{L} X\simeq A$ in $\mathbf{D}(A^e\mbox{-}{\rm Mod})$.
\end{enumerate}
\end{prop}

Following \cite[Definition~3.4]{Ric91}, such a complex $X$, viewed as an object in $\mathbf{D}(B\otimes A^{\rm op}\mbox{-Mod})$,  is called a \emph{two-sided tilting complex}, and the complex $Y$ above is called an \emph{inverse} of $X$, which is unique up to isomorphism in $\mathbf{D}(A\otimes B^{\rm op}\mbox{-Mod})$.

\begin{proof}
Recall that $X\otimes^\mathbb{L}_A- $ is left adjoint to  $\mathbb{R}{\rm Hom}_B(X, -)$. The implication ``(1) $\Leftrightarrow$ (2)" follows immediately. We infer ``(3) $\Rightarrow$ (1)" from the natural isomorphisms (\ref{iso:XY}) and (\ref{iso:YX}).

We now assume that (1) and (2) hold. To prove (3), we fix an injective resolution $\epsilon\colon B\rightarrow I$ of the $B$-$B$-bimodule $B$. The flatness of $B$ implies that $\epsilon$ is also an injective resolution of the one-sided module. For each complex $M$ of $B$-modules,  fix a dg-projective resolution $\pi_M\colon {\bf p}(M)\rightarrow M$ and a dg-injective resolution $\iota_M\colon M\rightarrow {\bf i}(M)$. We observe that
$$\epsilon\otimes_B {\rm Id}_{{\bf p}(M)}\colon {\bf p}(M)=B\otimes_B{\bf p}(M)\longrightarrow I\otimes_B {\bf p}(M)$$
is a quasi-isomorphism. Since ${\bf i}(M)$ is dg-injective, there exists a unique morphism $\vartheta_M$ in ${\bf K}(B\mbox{-Mod})$ making the following square commute.
\[
\xymatrix{
{\bf p}(M)\ar[d]_-{\pi_M} \ar[rr]^-{\epsilon\otimes_B {\rm Id}_{{\bf p}(M)}}  && I\otimes_B {\bf p}(M) \ar@{.>}[d]^-{\vartheta_M}\\
M\ar[rr]^-{\iota_M} && {\bf i}(M)
}\]
The uniqueness of $\vartheta_M$ implies that it is functorial in $M$. We observe that $\vartheta_M$ is a quasi-isomorphism. Since both ${\bf i}(B)$ and $I\otimes_B {\bf p}(B)=I$ are dg-injective resolutions of $B$, we infer that $\vartheta_B$ is even a homotopy equivalence.

 Consequently, we have a natural transformation
$$\alpha_M\colon {\rm Hom}_B(X, I)\otimes_B {\bf p}(M)\longrightarrow {\rm Hom}_B(X, I\otimes_B {\bf p}(M))\longrightarrow {\rm Hom}_B(X, {\bf i}(M)),$$
where the  morphism on the right is induced by $\vartheta_M$, and the  morphism on the left sends $f\otimes_B y$ to $(x\mapsto (-1)^{|x|\cdot |y|} f(x)\otimes_B y)$ for homogeneous elements $f\in {\rm Hom}_B(X, I)$, $y\in {\bf p}(M)$ and $x\in X$.

We might view $\alpha$ as a natural transformation
$$\alpha\colon {\rm Hom}_B(X, I)\otimes^\mathbb{L}_B-\longrightarrow \mathbb{R}{\rm Hom}_B(X, -).$$
Since $\vartheta_B$ is a homotopy equivalence, we infer that $\alpha_B$ is an isomorphism. The following full subcategory
$$\mathcal{L}=\{M\in \mathbf{D}(B\mbox{-Mod})\; |\; \alpha_M \mbox{ is an isomorphism}\}$$
is triangulated and contains $B$. As an equivalence, the functor $\mathbb{R}{\rm Hom}_B(X, -)$ commutes arbitrary coproducts. From this fact, we refer that $\mathcal{L}$ is localizing. By the well-known fact that ${\rm Loc}\langle B\rangle=\mathbf{D}(B\mbox{-Mod})$, we infer that $\mathcal{L}=\mathbf{D}(B\mbox{-Mod})$. In other words, $\alpha$ is a natural isomorphism.

Set $Y={\rm Hom}_B(X, I)$. Recall that $\mathbb{R}{\rm Hom}_B(X, -)$ is a quasi-inverse of $X\otimes_A^\mathbb{L}-$. We infer from the isomorphism $\alpha$ that both $(Y\otimes^\mathbb{L}_B X)\otimes^\mathbb{L}_A-$ and $(X\otimes^\mathbb{L}_A Y)\otimes^\mathbb{L}_B-$ are isomorphic to the identity functors. Then we deduce the required isomorphisms in (3) from Lemma~\ref{lem:bimod-A}.
\end{proof}

The following well-known notion is an unbounded version of the one in \cite[Definition~3.4]{Ric91}; compare \cite[Remark~1.10]{Yek}.

\begin{defn}[Rickard]
A derived equivalence $F\colon \mathbf{D}(A\mbox{-}{\rm Mod})\rightarrow \mathbf{D}(B\mbox{-}{\rm Mod})$ is called \emph{standard} if it is isomorphic to $X\otimes_A^\mathbb{L}-$ for some two-sided tilting complex $X$ of $B$-$A$-bimodules.
\end{defn}

The following facts are well known; see \cite[Proposition~4.1]{Ric91}.

\begin{lem}\label{lem:stan-comp}
The composition of two standard derived equivalences is standard, and any quasi-inverse of a standard derived equivalence is standard.\hfill $\square$
\end{lem}

The following result is inspired by the work \cite{Kel93}, and partially justifies the notion of canonical derived equivalence.

\begin{thm}\label{thm:stan=can}
Assume that $A$ and $B$ are flat algebras, and that $F\colon \mathbf{D}(A\mbox{-}{\rm Mod})\rightarrow \mathbf{D}(B\mbox{-}{\rm Mod})$ is a derived equivalence. Then $F$ is standard if and only if it is canonical.
\end{thm}

\begin{proof}
For the ``if" part, we take a tilting complex $P\in \mathbf{K}^b(B\mbox{-proj})$ and an algebra isomorphism $\phi\colon A\rightarrow {\rm End}_{\mathbf{K}^b(B\mbox{-}{\rm proj})}(P)^{\rm op}$. Assume that $F=\Psi_{(P,\phi)}$. Consider $\Gamma=(\tau_{\leq 0} \; {\rm End}_B(P))^{\rm op}$. The flatness of $B$ implies that $\Gamma$ is  dg-flat as a complex of $\mathbb{K}$-modules.   Recall from (\ref{equ:pi-P}) the surjective quasi-isomorphism $\pi_P\colon \Gamma\rightarrow {\rm End}_{\mathbf{K}^b(B\mbox{-}{\rm proj})}(P)^{\rm op}$. We will view  $A$ as a dg $\Gamma$-$A$-bimodule via $\phi^{-1}\circ \pi_P$. Therefore, we have a complex $P\otimes_\Gamma^\mathbb{L} A$ of $B$-$A$-bimodules. Indeed, we might use a dg-projective resolution ${\bf p}(A)$ as a dg $\Gamma$-$A$-bimodule, and set $P\otimes_\Gamma^\mathbb{L} A=P\otimes_\Gamma {\bf p}(A)$. Here, we use the flatness of $A$.

For any complex $M$ of $A$-modules, we have natural isomorphisms.
\begin{align*}
F(M)=P\otimes_\Gamma^\mathbb{L} M\simeq P\otimes_\Gamma^\mathbb{L}(A\otimes_A^\mathbb{L} M) \simeq (P\otimes_\Gamma^\mathbb{L} A)\otimes^\mathbb{L}_A M
 \end{align*}
 Here, the rightmost isomorphism uses the following fact implicitly: as a complex of right $A$-modules, ${\bf p}(A)$ is dg-flat. These isomorphisms imply that $F$ is standard.

 For the ``only if" part, we assume that $F$ is standard. Theorem~\ref{thm:factor} yields a factorization $F\simeq F_2F_1$ with $F_1$ a pseudo-identity on $\mathbf{D}(A\mbox{-Mod})$ and $F_2$ canonical. By the proof above, $F_2$ is standard. We have a natural isomorphism $F_1\simeq F_2^{-1}F$, which implies that the restriction $F_2^{-1}F$ on $\Lambda\mbox{-proj}$ is isomorphic to the inclusion $\Lambda\mbox{-proj}\rightarrow \mathbf{D}(\Lambda\mbox{-proj})$.

 By Lemma~\ref{lem:stan-comp}, $F_2^{-1}F$ is standard. Applying Lemma~\ref{lem:bimod-A} to $F_2^{-1}F$, we infer that $F_2^{-1}F$ is isomorphic to the identity. Therefore, $F$ is isomorphic to $F_2$, and is canonical.
\end{proof}

\begin{rem}
 Combining  Theorem~\ref{thm:stan=can} and Lemma~\ref{lem:stan-comp}, one answers Question~\ref{ques:1} affirmatively in the flat case.
\end{rem}

In view of  Theorem~\ref{thm:stan=can}, we mention that Conjecture~\ref{conj:can} is a natural extension of the following well-known conjecture, which is a unbounded version of the open question raised in \cite[Section~3]{Ric91}.

\begin{conj}[Rickard]\label{conj:Rickard}
Any derived equivalence between  flat algebras is standard.
\end{conj}

The following result is analogous to Theorem~\ref{thm:factor}.

\begin{prop}
Assume that $A$ and $B$ are flat algebras, and that $F\colon \mathbf{D}(A\mbox{-}{\rm Mod})\rightarrow \mathbf{D}(B\mbox{-}{\rm Mod})$ is a derived equivalence. Then we have two factorizations $$F\simeq F_2F_1\simeq F'_1F'_2$$
with $F_1$ a pseudo-identity on $\mathbf{D}(A\mbox{-}{\rm Mod})$, $F'_1$ a pseudo-identity on $\mathbf{D}(B\mbox{-}{\rm Mod})$ and $F_2, F'_2$ standard. Both factorizations are unique up to isomorphism.
\end{prop}

\begin{proof}
    The first factorization follows from Theorems~\ref{thm:factor} and ~\ref{thm:stan=can}. One proves the second one by applying the first factorization to a quasi-inverse $F$.
\end{proof}

\begin{lem}\label{lem:stan+ps}
Let $A$ and $B$ be two flat algebras which are derived equivalent. Then the following statements hold.
\begin{enumerate}
    \item Any derived equivalence $\mathbf{D}(A\mbox{-}{\rm Mod})\rightarrow \mathbf{D}(B\mbox{-}{\rm Mod})$ is standard.
   \item Any pseudo-identity on $\mathbf{D}(A\mbox{-}{\rm Mod})$ is isomorphic to the identity functor.
   \item Any pseudo-identity on $\mathbf{D}(B\mbox{-}{\rm Mod})$ is isomorphic to the identity functor.
\end{enumerate}
\end{lem}

\begin{proof}
In view of Theorem~\ref{thm:stan=can}, ``(1) $\Leftrightarrow$ (2)'' follows from Corollary~\ref{cor:can=ps}. Since any quasi-inverse of a standard derived equivalence is also standard, (1) is equivalent to the following statement: (1')  any derived equivalence $\mathbf{D}(B\mbox{-}{\rm Mod})\rightarrow \mathbf{D}(A\mbox{-}{\rm Mod})$ is standard. Applying Corollary~\ref{cor:can=ps} again, we have ``(1') $\Leftrightarrow$ (3)''.  This completes the proof.
\end{proof}

We verify Conjecture~\ref{conj:Rickard} in the hereditary case; it might be viewed as a unbounded version of \cite[Theorem~1.8]{MY}; compare Theorem~\ref{thm:hereditary+can}.

\begin{prop}\label{prop:hereditary+stan}
Assume that $A$ and $B$ are flat algebras, which are derived equivalent to a left hereditary  flat algebra $H$. Then any derived equivalence $\mathbf{D}(A\mbox{-}{\rm Mod})\rightarrow \mathbf{D}(B\mbox{-}{\rm Mod})$ is standard.
\end{prop}

\begin{proof}
By Proposition~\ref{prop:hereditary+ps}, any pseudo-identity on $\mathbf{D}(H\mbox{-}{\rm Mod})$ is isomorphic to the identity functor. Since $A$ and $H$ are derived equivalent, Lemma~\ref{lem:stan+ps} implies that  any pseudo-identity on $\mathbf{D}(A\mbox{-}{\rm Mod})$ is also isomorphic to the identity functor. Then the required statement follows immediately from the same lemma.
\end{proof}

\vskip 5pt

\noindent {\bf Acknowledgement.}\; The author thanks Professor Yuanyang Zhou for the reference \cite{Rou}, and  thanks Xiaofa Chen and Haoyu Wang for many helpful comments. The project is supported by National Natural Science Foundation of China (No.s 12325101, 12131015 and 12161141001).

\bibliography{}

\begin{thebibliography}{999}

\bibitem{BBD}{\sc  A. Beilinson, J. Bernstein, and P. Deligne,} {\em Faisceaux
pervers}, Ast\'{e}risque {\bf 100}, Soc. Math. France, Paris, 1982.




\bibitem{Bondal} {\sc A.I. Bondal}, {\em Representations of associative algebras and coherent sheaves}, Math. USSR Izv. {\bf 34} (1990), 23--42.

\bibitem{BK} {\sc A.I. Bondal, and M.M. Kapranov}, {\em Representable functors, Serre functors, and mutations}, Math. USSR Izv. {\bf 35} (3) (1990),  519--541.

\bibitem{Bon} {\sc M.V. Bondarko}, {\em Weight structures vs. t-structures; weight filtrations, spectral sequences, and complexes (for motives and in general)}, J. K-Theory {\bf  6} (3) (2010),  387--504.

\bibitem{BN} {\sc M. B\"{o}kstedt, and A. Neeman}, {\em Homotopy limits in triangulated categories}, Compos. Math. {\bf 86} (1993), 209--234.



\bibitem{CC} {\sc X. Chen, and X.W. Chen}, {\em Liftable derived equivalences and objective categories}, Bull. London Math. Soc. {\bf 52} (2020), 816--834.



\bibitem{CLZ} {\sc X.W. Chen, Z. Lin, and Y. Zhou}, {\em The extensions of t-structures}, Ark. Mat. {\bf 61} (2023), 323--342.




\bibitem{CY} {\sc X.W. Chen, and Y. Ye}, {\em The D-standard and K-standard categories}, Adv. Math. {\bf 333} (2018), 159--193.

\bibitem{Chen-Wang} {\sc X.W. Chen, and Z. Wang}, {\em  Differential graded enhancements of singularity categories}, preprint, 2023.

\bibitem{ChenYP} {\sc Y. Chen}, {\em Derived equivalences between matrix subrings and their applications}, J. Algebra {\bf 370} (2012), 113--132.

\bibitem{CPS} {\sc E. Cline, B. Parshall, and L. Scott}, {\em Derived categories and Morita theory}, J. Algebra {\bf 104} (1986), 397--409.

\bibitem{Dri}{\sc V. Drinfeld}, {\em DG quotients of DG categories}, J. Algebra {\bf 272} (2004), 643--691.

\bibitem{Huy} {\sc  D. Huybrechts}, Fourier-Mukai Transformations in Algebraic Geometry, Oxford Math. Monogr., Clarendon Press, Oxford, 2006.


\bibitem{JK} {\sc P. Jorgensen, and K. Kato}, {\em Triangulated subcategories of extensions, stable t-structures, and triangles of recollements}, J. Pure Appl. Algebra {\bf 219} (12) (2015), 5500--5510.

\bibitem{Kel93} {\sc B. Keller}, {\em A remark on tilting theory and DG algebras}, Manuscripta Math. {\bf 79} (3-4) (1993), 247--252.

\bibitem{Kel94}{\sc B.~Keller}, {\em Deriving DG categories,} Ann. Sci. \'{E}cole Norm. Sup. (4) {\bf 27} (1994), no. 1, 63--102.

\bibitem{Kel00} {\sc B. Keller}, {\em Bimodule complexes via strong homotopy actions}, Algebr. Represent. Theor. {\bf 3} (2000), 357--376.


\bibitem{KZ} {\sc S. Koenig, and A. Zimmermann}, {\em Tilting hereditary orders}, Comm. Algebra {\bf 24} (6) (1996), 1897--1913.



\bibitem{Kra} {\sc H. Krause}, {\em Derived categories, resolutions and Brown representability}, in: Interactions between Homotopy Theory and Algebra, 101--139, Contemp. Math. {\bf 436}, Amer. Math. Soc., Providence, RI, 2007.




\bibitem{Miy} {\sc J.I. Miyachi},  Derived Categories with Applications to Representations of Algebras, Chiba Univ., 2000, available at: http://www.u-gakugei.ac.jp/$\sim$miyachi/seminar.html.

\bibitem{MY} {\sc  J.I. Miyachi, and A. Yekutieli}, {\em Derived Picard groups of finite-dimensional hereditary algebras}, Compos. Math. {\bf 129} (2001), 341--368.

\bibitem{Pau} {\sc D. Pauksztello}, {\em Compact corigid objects in triangulated categories and co-t-structures}, Cent. Eur. J. Math. {\bf 6} (2008), 25--42.


\bibitem{PS} {\sc L. Positeselski, and M. Schnuerer}, {\em Unbounded derived categories of small and big modules: Is the natural functor fully faithful?} J. Pure Appl. Algebra {\bf 225} (11) (2021), 106722.



\bibitem{Ric89} {\sc  J. Rickard}, {\em Morita theory for derived categories}, J. Lond. Math. Soc. (2) {\bf 39} (1989), 436--456.

\bibitem{Ric91} {\sc J. Rickard}, {\em Derived equivalences as derived functors}, J. Lond. Math. Soc. (2) {\bf 43} (1991), 37--48.

\bibitem{Ric96} {\sc J. Rickard}, {\em Splendid equivalences: derived categories and permutation modules},
Proc Lond. Math. Soc. {\bf 3} (2) (1996), 331--358.

\bibitem{Rou} {\sc R. Rouquier}, {\em Block theory via stable and Rickard equivalences}, in: Modular Representation Theory of Finite Groups, New York De Gruyter, 2001, 101--146.


\bibitem{Toen} {\sc  B. T\"{o}en}, {\em The homotopy theory of dg-categories and derived Morita theory}, Invent. Math. {\bf 167} (2007),  615--667.


\bibitem{Xi} {\sc C. Xi}, {\em Derived equivalences of algebras}, Bull. London Math. Soc. {\bf 50} (2018), 945--985.

\bibitem{Yek} {\sc A. Yekutieli}, {\em Dualizing complexes, Morita equivalence and the derived Picard group of a ring}, J. London Math. Soc. (2) {\bf 60} (1999),  723--746.

\bibitem{Zim} {\sc A. Zimmermann}, Representation Theory, A Homological Algebra Point of View, Springer International Publishing, Switzerland, 2014.

\end{thebibliography}

\vskip 10pt

 {\footnotesize \noindent Xiao-Wu Chen \\
 Key Laboratory of Wu Wen-Tsun Mathematics, Chinese Academy of Sciences,\\
 School of Mathematical Sciences, University of Science and Technology of China, Hefei 230026, Anhui, PR China\\
 E-mail: xwchen$\symbol{64}$mail.ustc.edu.cn}

\end{document}